\documentclass[11pt]{article}
\usepackage[top=1in, bottom=1in, left=1in, right=1in]{geometry}
\usepackage[linktocpage,colorlinks,linkcolor=blue,anchorcolor=blue,citecolor=blue,urlcolor=blue]{hyperref}
\usepackage{graphicx} % Required for inserting images
\usepackage{hyperref}       % hyperlinks
\usepackage{url}            % simple URL typesetting
\usepackage{booktabs}       % professional-quality tables
\usepackage{amsfonts}       % blackboard math symbols
\usepackage{nicefrac}       % compact symbols for 1/2, etc.
\usepackage{microtype}      % microtypography
\usepackage{xcolor}
\usepackage{caption,graphicx,newfloat}
\DeclareCaptionType{BOX}
\usepackage{amsmath}
\usepackage{amssymb}
\usepackage{multirow}
\usepackage{enumerate}
\usepackage{arydshln}
\usepackage{multicol}
\usepackage{algorithm}
\usepackage{algorithmic}
\usepackage{amsthm}
\usepackage{wrapfig}
\usepackage{subcaption}
\usepackage[T1]{fontenc}
\usepackage{lmodern}
\usepackage{setspace}

\usepackage[capitalize,noabbrev]{cleveref}
\usepackage{breakcites}

\def\x{\bar{x}}
\def\y{\bar{y}}
\def\v{\bar{v}}
\def\d{\bar{d}}
\def\t{\bar{t}}
\def\argmin{\textrm{argmin}}
\def\etal{et al. }

\newtheorem{theorem}{Theorem}[section]

\newtheorem{remark}{Remark}[section]

\newtheorem{lemma}{Lemma}[section]

\newtheorem{assumption}{Assumption}[section]

\title{A Single-Loop Algorithm for Decentralized Bilevel Optimization}

\date{}

\begin{document}
\author{ Youran Dong\footnotemark[1],\quad Shiqian Ma\footnotemark[2],\quad Junfeng Yang\footnotemark[1],\quad Chao Yin\footnotemark[1]}

\renewcommand{\thefootnote}{\fnsymbol{footnote}}

\footnotetext[1]{Department of Mathematics, Nanjing University, Nanjing, P. R. China. Email: yrdong@smail.nju.edu.cn, jfyang@nju.edu.cn, yinchao@smail.nju.edu.cn. Research supported by the National Natural Science Foundation of China (NSFC-12371301).}
\footnotetext[2]{Department of Computational Applied Mathematics and Operations Research, Rice University, Houston, USA. Email: sqma@rice.edu. Research supported in part by NSF grants DMS-2243650, CCF-2308597, CCF-2311275 and ECCS-2326591, and a startup fund from Rice University.}

\maketitle

\begin{abstract}
Bilevel optimization has gained significant attention in recent years due to its broad applications in machine learning. This paper focuses on bilevel optimization in decentralized networks and proposes a novel single-loop algorithm for solving decentralized bilevel optimization with a strongly convex lower-level problem. Our approach is a fully single-loop method that approximates the hypergradient using only two matrix-vector multiplications per iteration. Importantly, our algorithm does not require any gradient heterogeneity assumption, distinguishing it from existing methods for decentralized bilevel optimization and federated bilevel optimization. Our analysis demonstrates that the proposed algorithm achieves the best-known convergence rate for bilevel optimization algorithms. We also present experimental results on hyperparameter optimization problems using both synthetic and MNIST datasets, which demonstrate the efficiency of our proposed algorithm.
\end{abstract}

\section{Introduction}
Bilevel optimization (BO) has received increasing attention in recent studies due to its wide applications in machine learning, including but not limited to hyperparameter optimization \cite{pedregosa2016hyperparameter, franceschi2018bilevel}, meta learning \cite{franceschi2018bilevel, rajeswaran2019meta, ji2020convergence} and adversarial training \cite{bishop2020optimal, wang2021fast, zhang2022revisiting}. A generic BO takes the form
\begin{align}\label{BO}
\mathop{\min}\limits_{x \in \mathbb{R}^p}  \Phi(x) = F(x,y^*(x)),\quad \mathrm{s.t.~} \   y^*(x)=\arg\min_{y\in \mathbb{R}^q} f(x,y).
\end{align}
Throughout this paper, we assume that the lower-level (LL) objective function $f$ is twice continuously differentiable and strongly convex with respect to $y$ for any fixed $x$. Problem \eqref{BO} aims at minimizing the upper-level (UL) function $F$ with respect to $x$ with $y$ being the optimal solution of the LL problem for fixed $x$. Algorithms for solving BO \eqref{BO} have been studied extensively. When the LL problem is strongly convex with respect to $y$ so that it admits unique solution for fixed $x$, a natural idea to solve \eqref{BO} is to apply gradient descent for the UL problem. Under the assumption that $F$ is smooth, the gradient descent method for solving \eqref{BO} updates the iterate as follows:
\begin{equation*}
x^{k+1} := x^k - \tau_{x,k} \nabla\Phi(x^k),
\end{equation*}
where $\tau_{x,k}>0$ is a step size, and the hypergradient $\nabla\Phi(x)$ is given by
\begin{equation}\label{BO-hypergradient}
\nabla\Phi(x) := \nabla_1 F(x,y^*(x)) - \nabla_{12}^2f(x,y^*(x))[\nabla_{22}^2f(x,y^*(x))]^{-1}\nabla_2F(x,y^*(x)).
\end{equation}
Two challenges arise from computing the hypergradient in \eqref{BO-hypergradient}: (i) how to efficiently (approximately) compute $y^*(x^k)$, which requires solving the LL problem for given $x^k$; and (ii) how to deal with the matrix inversion, or equivalently, solve the linear system in \eqref{BO-hypergradient}.
Different approaches addressing these two issues lead to different algorithms for solving \eqref{BO}. Let $K$ and $T$ be given positive integers. A basic algorithm along this line for solving \eqref{BO} updates the iterates as follows:
\begin{align}
& \mbox{for } k = 0,1,\ldots, K-1 \nonumber\\
& \qquad y^{k,0} = y^{k-1,T} \nonumber\\
& \qquad \mbox{for } t = 0,1,\ldots, T-1 \nonumber\\
& \qquad\qquad y^{k,t+1} = y^{k,t} -\tau_{y,k} \nabla_2 f(x^k,y^{k,t}) \label{BO-double-loop}\\
& \qquad \mbox{end for}\nonumber  \\
& \qquad x^{k+1} = x^k -  \tau_{x,k} \widetilde{\nabla \Phi}(x^k)\nonumber \\
& \mbox{end for} \nonumber
\end{align}
where 
$\widetilde{\nabla \Phi}(x^k)$ is an approximation of the hypergradient $\nabla \Phi(x^k)$ and is defined as
\begin{align}\label{BO-hypergradient-linear-system}
\widetilde{\nabla \Phi}(x^k) = \nabla_1F(x^k, y^{k,T}) - \nabla_{12}^2f(x^k, y^{k,T})[\nabla_{22}^2f(x^k,y^{k,T})]^{-1}\nabla_2F(x^k, y^{k,T}).
\end{align}
In practice, exactly calculating the Hessian inverse or solving the linear system of equations in \eqref{BO-hypergradient-linear-system} is computationally inefficient, and hence two representative approaches to estimate \eqref{BO-hypergradient-linear-system} have been proposed in the literature: iterative differentiation (ITD) and approximate implicit differentiation (AID).
Approaches related to ITD, such as those proposed in \cite{franceschi2018bilevel, shaban2019truncated, ji2020convergence, grazzi2020iteration, ji2021bilevel, ji2022will}, leverage automatic differentiation to approximate the hypergradient using backpropagation.
Approaches related to AID, including those proposed in
\cite{pedregosa2016hyperparameter, ghadimi2018approximation, grazzi2020iteration, ji2021bilevel, ji2022will, chen2021closing, hong2023two, dagreou2022framework}, use various methods to 
approximately evaluate $[\nabla_{22}^2f(x^k,y^{k,T})]^{-1}\nabla_2F(x^k, y^{k,T})$ in 
\eqref{BO-hypergradient-linear-system}. Some of these methods employ gradient descent or conjugate gradient methods, while others use Neumann series to approximate the Hessian inverse.
Additionally, it should be noted that the update scheme \eqref{BO-double-loop} involves a double-loop structure, where updating $x$ constitutes the outer loop while updating $y$ represents the inner loop. However, this structure is not preferable in practical settings. Some works \cite{chen2021closing, hong2023two} eliminate this double-loop structure by taking $T=1$ in \eqref{BO-double-loop}, yet still require the use of the AID approach to estimate \eqref{BO-hypergradient-linear-system}.
Furthermore, both AID and ITD approaches need another sub-loop to estimate \eqref{BO-hypergradient-linear-system}, which may involve $\Theta(\log K)$ Hessian- and Jacobian-vector multiplications per iteration (see, e.g., \cite[Theorem 2]{ji2021bilevel}). 

Recently, Dagr{\'e}ou \etal \cite{dagreou2022framework} proposed a fully single-loop framework (named SOBA) for solving \eqref{BO} that only needs one matrix-vector multiplication to approximately solve the linear system  of equations in \eqref{BO-hypergradient-linear-system} in each iteration. The SOBA algorithm maintains three sequences and updates them as
\begin{align}\label{soba}
y^{k+1} = y^k - \beta_k D_y^k, \quad
v^{k+1} = v^k + \eta_k D_v^k, \quad
x^{k+1} = x^k - \alpha_k D_x^k,
\end{align}
where $\alpha_k$, $\beta_k$ and $\eta_k$ are stepsizes, $D_y^k, D_v^k$ and $D_x^k$ are respectively unbiased stochastic estimators of $d_y(x^k,y^k),d_v(x^k,y^k,v^k)$ and $d_x(x^k,y^k,v^k)$ defined as $d_y(x,y) = \nabla_2f(x,y)$,
\[
d_v(x,y,v) = \nabla_2F(x,y) - \nabla^2_{22}f(x,y)v \text{~~and~~}
d_x(x,y,v) = \nabla_1F(x,y) - \nabla^2_{12}f(x,y)v.
\]
The SOBA framework \cite{dagreou2022framework} has been extended by \cite{liu2023averaged} to handle the case where the LL problem is merely convex. In this case, it is  assumed that the sequence $\{v^k\}$ is bounded to ensure convergence.
Other notable BO algorithms include \cite{kwon2023fully,chen2023near} which first convert \eqref{BO} to an equivalent constrained single-level problem, and then approximately solve this reformulated problem. While this approach circumvents the need for computing the approximate hypergradient \eqref{BO-hypergradient-linear-system}, it requires careful consideration when handling the constraints. In \cite{hu2023improved}, the authors suggest transforming BO into an unconstrained constraint dissolving problem, enabling the direct application of efficient unconstrained optimization methods to BO problems. 

The main focus of this paper is to design a single-loop algorithm for decentralized bilevel optimization (DBO). DBO considers BO in a decentralized network, where the data are naturally distributed to $n$ agents, each with access to their own local data and communication limited to their immediate neighbors in the network. The $n$ agents cooperatively solve the BO problem through local updates and communications. Decentralized optimization has many benefits, such as enhancing computational efficiency and sharing data while protecting data privacy \cite{lian2017can}.
The general form of DBO is given below:
\begin{align}\label{DBO}
\mathop{\min}\limits_{x \in \mathbb{R}^p}  &\Phi(x) = F(x,y^*(x)):= \frac{1}{n}\sum_{i=1}^{n}F_i(x,y^*(x))\quad\nonumber\\
\mathrm{s.t.} \   &y^*(x)=\arg\min_{y\in \mathbb{R}^q} f(x,y):=\frac{1}{n}\sum_{i=1}^{n}f_i(x,y),
\end{align}
where the $i$-th agent only has access to the data related to $F_i$ and $f_i$. 
To illustrate DBO, let us consider the example of training a binary classification model using medical data from $n$ different hospitals. Suppose we want to train a logistic regression model that can predict whether a tumor is benign or malignant based on medical data (such as MRI). The hospitals are local agents that cannot share their data with other hospitals to preserve patients' privacy. In this case, the hyperparameter tuning problem can be formulated as follows, and we need to design a decentralized algorithm to solve it in a distributed network:
\begin{align*}%\label{DBO_h}
\mathop{\min}\limits_{\lambda}  \frac{1}{n}\sum_{i=1}^{n}\sum_{(x_e,y_e)\in D^{\prime}_i}\psi(y_e x_e^{\top}\omega^*(\lambda)), \quad
\mathrm{s.t.} \   \omega^*(\lambda)=\arg\min_{\omega} \frac{1}{n}\sum_{i=1}^{n}\sum_{(x_e,y_e)\in D_i}\psi(y_e x_e^{\top}\omega)+\lambda \|\omega\|_2^2.
\end{align*}
In this problem, $\psi$ is a loss function used to train the logistic regression model. A commonly used loss function is the logistic loss, given by $\psi(t)=\log (1+e^{-t})$. 
The parameter vector $\omega$ comprises the regression parameters, while $\lambda$ represents the hyperparameter of the $\ell_2$-norm-squared regularizer, which is used to prevent overfitting.
The training and testing datasets from hospital $i$ are denoted by $D_i$ and $D^{\prime}_i$, respectively. 

The main challenge in designing a decentralized gradient method for solving the DBO \eqref{DBO} is how to compute the hypergradient. Note that, the hypergradient $\nabla \Phi(x)$ of \eqref{DBO} is given by
\begin{align}\label{hg}
\nabla\Phi(x)&=\frac{1}{n}\sum_{i=1}^{n}\nabla_1 F_i\left(x, y^*(x)\right) \nonumber\\
&\quad -\left[\frac{1}{n}\sum_{i=1}^{n}\nabla^2_{12} f_i\left(x, y^*(x)\right)\right]\left[\frac{1}{n}\sum_{i=1}^{n}\nabla_{22}^2 f_i\left(x, y^*(x)\right)\right]^{-1} \frac{1}{n}\sum_{i=1}^{n}\nabla_2 F_i\left(x, y^*(x)\right).
\end{align}
Calculation of the hypergradient \eqref{hg} is not possible through a single agent, and instead requires cooperative computation among all agents through communication.
The first algorithm for solving DBO \eqref{DBO} was due to \cite{chen2022decentralized}, where the authors proposed the DSBO algorithm that incorporates a decentralized algorithm to solve the linear system of equations in \eqref{hg}. The per-iteration complexity of the DSBO algorithm was later improved by the same authors by employing the moving average technique \cite{chen2023decentralized}. In \cite{lu2022stochastic}, the authors proposed a stochastic linearized augmented Lagrangian method (SLAM) for solving DBO \eqref{DBO}. Another type of distributed BO, Federated BO, has also been studied in the literature. For example, Yang \etal \cite{yang2023simfbo} proposed the SimFBO algorithm for solving \eqref{DBO} but in a federated network.
All these algorithms for decentralized BO and federated BO require certain gradient heterogeneity in order to guarantee the convergence or lower per-iteration complexity. However, this kind of assumption is very strong and may not hold in certain scenarios (see, e.g., \cite{pu2021distributed}). We list below the heterogeneity assumptions in these papers.
\begin{itemize}
	\item (DSBO,  \cite[Assumption 2.4]{chen2022decentralized}) Assume the data associated with $f_i$ is independent and identically distributed, $i=1,\ldots,n.$
	\item (MA-DSBO, \cite[Assumption 2.3]{chen2023decentralized}) For all $i$, there exists a constant $\delta \geq 0$ such that
	\begin{align*}
	\left\|\nabla_2 f_i(x, y)-\frac{1}{n} \sum\nolimits_{j=1}^n \nabla_2 f_j(x, y)\right\| \leq \delta, \quad \forall x,y.
	\end{align*}
	\item (SLAM,  \cite[Theorem 1]{lu2022stochastic}) For all $i$, there exists a constant $L \geq 0$ such that
	\begin{align*}
	\left\|\nabla^2_{22}f_i(x,y)-\frac{1}{n}\sum\nolimits_{j=1}^{n}\nabla^2_{22}f_j(x,y')\right\| \leq L \|y-y^{\prime}\|, \quad \forall x,y,y^{\prime}.
	\end{align*}
	It should be noted that if $y=y^{\prime}$, this assumption becomes $\nabla^2_{22}f_i(x,y)=\frac{1}{n}\sum_{j=1}^{n}\nabla^2_{22}f_j(x,y)$.
	\item (SimFBO, \cite[Assumption 4]{yang2023simfbo}) There exist constants $\delta_1\geq 1$ and $\delta_2\geq 0$ such that
	\begin{align*}
	\frac{1}{n}\sum_{i=1}^n \left\|\nabla_2 f_i(x, y)\right\|^2 \leq \delta_1^2\left\|\frac{1}{n}\sum\nolimits_{i=1}^n \nabla_2 f_i(x, y)\right\|^2+\delta_2^2, \quad \forall x,y.
	\end{align*}
\end{itemize}
These assumptions indicate the level of similarity between the local and the global objective functions. Our proposed algorithm, however, with the help of gradient tracking and projection technique, does not need any heterogeneity assumptions like these. 
Here we emphasize that in this paper we focus on deterministic DBO, because it is already very challenging. Extending the algorithms in this paper to stochastic DBO is definitely another very important task, and we leave it to a future work.

\paragraph{Main contributions.}
Our contributions in this work lie in several folds. First, we propose a single-loop algorithm for DBO, which has two main features: (i) it is of a single-loop structure; and (ii) it only needs two matrix-vector multiplications in each iteration. Second, we provide a convergence rate analysis for the proposed algorithm in the absence of any heterogeneity assumptions. This is in sharp contrast to existing works on decentralized BO and federated BO. Third, we demonstrate the great potential of our algorithm through numerical experiments on hyperparameter optimization.

\paragraph{Notation.}  We denote the optimal value of \eqref{DBO} as $F^*$. The gradients of $f$ with respect to $x$ and $y$ are denoted as $\nabla_1f(x,y)$ and $\nabla_2f(x,y)$ respectively, while the Jacobian matrix of $\nabla_1f$ and Hessian matrix of $f$ with respect to $y$ are denoted as $\nabla^2_{12}f(x,y)$ and $\nabla^2_{22}f(x,y)$ respectively.	If there is no further specification, it is assumed that $\|\cdot\|$ denotes the $\ell_2$ norm for vectors and the Frobenius norm for matrices. The operator norm of a matrix $Z$ is denoted by $\|Z\|_{\textrm{op}}$.

\section{A Single-Loop Algorithm for Decentralized Bilevel Optimization} \label{sec:alg}
In this section, we propose our single-loop algorithm for DBO (SLDBO). Its convergence results are given in Section \ref{sec:convergence}. Before presenting the algorithm, we specify our assumptions.

\subsection{Assumptions}\label{pre}
Throughout this paper, we adopt the following standard assumptions, which are commonly used in existing literature on bilevel optimization and decentralized optimization. For instance, Assumption \ref{Assump1} has been employed in previous works such as
\cite{ghadimi2018approximation, ji2021bilevel, chen2022decentralized,chen2023decentralized, ji2022will}, and Assumption \ref{Assump2} has been utilized in \cite{qu2017harnessing,nedic2017achieving,choi2023convergence,chen2022decentralized,chen2023decentralized,lu2022stochastic}. %These assumptions are established as standard in the literature and are taken as the basis for this paper.

\begin{assumption}\label{Assump1}
	The following assumptions hold for functions $F_i$ and $f_i$, $i=1,\ldots,n$, in \eqref{DBO}.
	\begin{enumerate}[(a)]
		\item For any fixed $x$, $f_i(x,\cdot)$ is $\sigma$-strongly convex, with $\sigma>0$ being a constant.

		\item The function $F_i$ is Lipschitz continuous with a Lipschitz constant of $L_{F,0}$, and its gradient $\nabla F_i$ is also Lipschitz continuous with a Lipschitz constant of $L_{F,1}$.

        \item The function $f_i$ is twice differentiable, with its gradient $\nabla f_i$ being Lipschitz continuous and having a Lipschitz constant of $L_{f,1}$. Moreover, the Hessian of $f_i$, denoted by $\nabla^2 f_i$, is also Lipschitz continuous with a Lipschitz constant of $L_{f,2}$.
	\end{enumerate}
\end{assumption}

\begin{assumption}[Network topology]\label{Assump2}
 Suppose the communication network is represented by a nonnegative weight matrix $W = (w_{ij})\in \mathbb{R}^{n\times n}$,
 where $w_{ij}=0$ if $i\neq j$ and agents $i$ and $j$ are not connected.
 Moreover, we assume that $W$ is symmetric and doubly stochastic, i.e. $W = W^{\mathsf{T}}$ and $W\mathbf{1_n} = \mathbf{1_n}$, where $\mathbf{1_n}$ is the all-one vector in $ \mathbb{R}^{n}$.
 Furthermore, the eigenvalues of $W$ satisfy $1 = \lambda_1> \lambda_2\geq \cdots \geq \lambda_n$ and $\rho := \max\{|\lambda_2|, |\lambda_n|\} < 1$.
\end{assumption}

\subsection{The Proposed SLDBO Algorithm}
Our goal is to extend the idea of SOBA \cite{dagreou2022framework} to handle the DBO \eqref{DBO}, which presents a significant challenge, especially without any heterogeneity assumptions. To tackle this, we propose to project $v^k$ onto a Euclidean ball with a pre-defined radius. By combining this with the gradient tracking technique, we successfully eliminate all heterogeneity assumptions.\footnote{Here we point out that shortly after we released the first version of our paper on arXiv, another work \cite{zhang2023communication} was uploaded to arxiv. The first version of \cite{zhang2023communication} extended SOBA to solving the stochastic DBO, but it did not provide any proof for the theoretical results. The second version of \cite{zhang2023communication} appeared on arxiv a few months later where the authors changed their algorithm in the first version to a different algorithm and a convergence analysis was provided for this new algorithm. We note that this new algorithm incorporated a projection step similar to the one used in our Algorithm \ref{alg:slDB}, and this projection step was missing in the algorithm in the first version of \cite{zhang2023communication}. This confirms the crucial role of our projection technique for a provably convergent algorithm for DBO.}
Before introducing our SLDBO algorithm, we define the following constants:
\begin{equation}\label{def-constants-main-text}
\left\{
\begin{array}{l}
r_v  := L_{F,0}/\sigma, \quad L_v:=\left(L_{F,1}+L_{f,2}r_v\right)\left(1+ L_{f,1}/\sigma \right),\quad L_1:=L_{F,1}+L_{f,2}r_v, \smallskip\\	
L_{\Phi}  :=L_{F, 1}+\frac{2 L_{F, 1} L_{f, 1}+L_{f, 2} L_{F, 0}^2}{\sigma}+\frac{2 L_{f, 1} L_{F, 0} L_{f, 2}+L_{f, 1}^2 L_{F, 1}}{\sigma^2}+\frac{L_{f, 2} L_{f, 1}^2 L_{F, 0}}{\sigma^3},  
\end{array}
\right.
\end{equation}
where constants such as $\sigma$, $L_{F,0}$, $L_{F,1}$, $L_{f,1}$, $L_{f,2}$ are defined in Assumptions \ref{Assump1}. We also define the projection operator $\mathcal{P}_{r}$ as
$\mathcal{P}_{r}[z]:= \arg\min_{\|z'\|\leq r}\|z'-z\|$,  which projects a given point onto the Euclidean ball with radius $r\geq 0$.

The details of our proposed algorithm, SLDBO, are presented in Algorithm \ref{alg:slDB}, while the setup of its initial points is shown in BOX \ref{box:initial}.
\vspace{2mm}

\noindent\fbox{
    \parbox{\textwidth}{
    \begin{itemize}
        \item Initial points $d_{x,i}^{-1} = d_{y,i}^{-1}=d_{v,i}^{-1}=t_{x,i}^{-1}=t_{y,i}^{-1}=t_{v,i}^{-1}=0$ $(i=1,2,\ldots,n)$.
        \item Initial points $x^{-1}_i = x^{-1}_j$, $y^{-1}_i = y^{-1}_j$, $v^{-1}_i = v^{-1}_j$, $i,j\in \{1,2,\ldots,n\}$, satisfying $\|v^{-1}_i\|\leq r_v$, where $r_v$ is defined in \eqref{def-constants-main-text}.        Note that the inequality can easily be satisfied by choosing $v_i^{-1}=0$, for instance.
        \item Compute initial points $x^0_i  =\sum_{j=1}^n w_{i j} (x_j^{-1}-\alpha t_{x, j}^{-1})=x^{-1}_i$,
        \[
        y^0_i   =\sum\nolimits_{j=1}^n w_{i j} (y_j^{-1}-\beta t_{y, j}^{-1})=y^{-1}_i
        \text{~~and~~}
        v^0_i   =\mathcal{P}_{r_v}\left[\sum\nolimits_{j=1}^n w_{i j} (v_j^{-1}+\eta t_{v, j}^{-1})\right]=v^{-1}_i,
        \]
        for $i=1,2,\ldots,n$, 
        where $ \alpha$, $\beta$ and $\eta $ are constant stepsizes.
 \end{itemize}
 \vspace{-0.3cm}
    \captionof{BOX}{Initialization of Algorithm \ref{alg:slDB}. \label{box:initial}}}
}

\begin{algorithm}[h]
	\caption{A Single-Loop Algorithm for DBO (SLDBO)}\label{alg:slDB}
	\begin{algorithmic}
             \STATE {\bf Input:} Let $K$ be the maximum iteration number and $r_v$ be defined in \eqref{def-constants-main-text}. Set initial points as in BOX \ref{box:initial}, as well as constant step sizes $ \alpha,\, \beta,\, \eta > 0$ satisfying the upper bounds in \eqref{eq:stepsize}. \;
		\FOR{$k=0,1,\ldots,K-1$}
		\FOR{$i=1,\ldots,n$}
		\STATE
  \vspace{-0.5cm}
		\begin{align}
        &d_{y,i}^k=\nabla_2 f_i(x_i^k,y_i^k);\label{yd}\\
        &d_{v,i}^k=\nabla_2 F_i(x_i^k,y_i^k) - \nabla^2_{22} f_i(x_i^k,y_i^k) v_i^k;\label{vd}\\
        &d_{x,i}^k=\nabla_1 F_i(x_i^k,y_i^k) - \nabla^2_{12} f_i(x_i^k,y_i^k) v_i^k;\label{xd}\\
		&t_{y, i}^k=\sum\nolimits_{j=1}^n w_{i j} t_{y, j}^{k-1}+d_{y,i}^k-d_{y, i}^{k-1},\quad
		y_i^{k+1}=\sum\nolimits_{j=1}^n w_{i j} (y_j^k-\beta  t_{y, j}^k); \label{yupdate}\\
		&t_{v, i}^k=\sum\nolimits_{j=1}^n w_{i j} t_{v, j}^{k-1}+d_{v,i}^k-d_{v, i}^{k-1},\quad v_i^{k+1}=\mathcal{P}_{r_v}\left[\sum\nolimits_{j=1}^n w_{i j}( v_j^k+\eta  t_{v, j}^k)\right];\label{vupdate}\\
		&t_{x, i}^k=\sum\nolimits_{j=1}^n w_{i j} t_{x, j}^{k-1}+d_{x,i}^k-d_{x, i}^{k-1},\quad x_i^{k+1}=\sum\nolimits_{j=1}^n w_{i j} (x_j^k-\alpha t_{x, j}^k).\label{xupdate}
		\end{align}
		\ENDFOR
		\ENDFOR
	\end{algorithmic}
\end{algorithm}

\begin{remark} \label{remark:alg-sldbo}

Some remarks on the SLDBO (Algorithm \ref{alg:slDB}) are in demand.

\begin{enumerate}[(i)]

    \item The $d_{y,i}^k$, $d_{v,i}^k$ and $d_{x,i}^k$ in \eqref{yd}-\eqref{xd} are the decentralized counterparts of $D_y^k$, $D_v^k$ and $D_x^k$ in SOBA \eqref{soba}.
    
    \item The updates for $y_i^{k+1}$, $v_i^{k+1}$, and $x_i^{k+1}$ outlined in \eqref{yupdate}-\eqref{xupdate} are based on similar ideas as in SOBA \eqref{soba}. However, since we are now addressing the decentralized problem \eqref{DBO}, we require communication steps using the communication matrix $W=(w_{ij})$. As it is introduced in Assumption \ref{Assump2}, if $w_{ij} > 0$, agent $j$ communicates with agent $i$, and the data communicated is multiplied by the weight $w_{ij}$. If $w_{ij}=0$, then there is no communication between the two agents. In this regard, we employ the adapt-then-combine diffusion strategy, which has been demonstrated to perform better in practical numerical experiments \cite{sayed2014adaptation}, rather than the combine-then-adapt diffusion strategy.
    
    \item  \label{re:proj} The updates for $t_{y,i}^k$, $t_{v,i}^k$, and $t_{x,i}^k$ described in \eqref{yupdate}-\eqref{xupdate} rely on the gradient tracking technique, a commonly employed strategy in distributed optimization literature \cite{di2016next,qu2017harnessing,nedic2017achieving}  and DBO literature \cite{chen2022decentralized,chen2023decentralized,gao2023convergence}. 
    Tracking sequences have been employed to eliminate heterogeneity assumptions for single-level problems. However, for DBO problems, an additional projection step is required in \eqref{vupdate} to accomplish this task. Specifically, the projection step ensures that the sequence $ \{v_i^k\} $ remains bounded, which is crucial for bounding the terms 
    $d_{x,i}^{k+1} - d_{x,i}^{k} $ and $ d_{v,i}^{k+1} - d_{v,i}^{k} $ in gradient tracking steps. We refer to Lemma \ref{dk+1-dk} in the Appendix for more discussions on this point.\end{enumerate}
\end{remark}

\section{Convergence Rate Results for SLDBO}\label{sec:convergence}
In this section, we provide the convergence results for Algorithm \ref{alg:slDB}.
Firstly, we construct a Lyapunov function, with the terms related to the consensus error. Then, by suitably selecting the parameters, we establish a descent property of the Lyapunov function, which further results in the $O(1/K)$ convergence rate for some stationarity measure. Our primary convergence rate results for Algorithm \ref{alg:slDB} are summarized in Theorem \ref{cr}.

\begin{theorem}\label{cr}
	For any integer $K\geq 1$, when $0\leq k\leq K$, define $\x^k = \frac{1}{n}\sum_{i=1}^{n}x^k_i$, $\y^k = \frac{1}{n}\sum_{i=1}^{n}y^k_i$ and $\v^k = \frac{1}{n}\sum_{i=1}^{n}v^k_i$.
	The following convergence rate results hold for Algorithm \ref{alg:slDB}.
	\begin{description}
		\item[(a) Stationarity.] For any integer $K\geq 1$, there holds
		$\min_{0 \le k \le {K-1}}  \|\nabla \Phi(\x^k)\|^2 = O\big(1/K\big)$.
		\item[(b) Consensus Error.]  For any integer $K\geq 1$, we have $	\min_{0 \le k \le {K-1}}\frac{1}{n}\sum_{i=1}^{n}\|x^k_i-\x^k\|^2=O\left(1/K\right)$, $\min_{0 \le k \le {K-1}}\frac{1}{n}\sum_{i=1}^{n}\|y^k_i-\y^k\|^2 =O\left(1/K\right)$ and $\min_{0 \le k \le {K-1}}\frac{1}{n}\sum_{i=1}^{n}\|v^k_i-\v^k\|^2=O\left(1/K\right)$.
	\end{description}
\end{theorem}

As per Part (a) of Theorem \ref{cr}, the convergence rate for stationarity is sublinear, with a rate of $O(1/K)$. This result is in line with the findings presented in \cite{ji2022will}. As far as we know, this is the currently best-known convergence rate result for both BO and DBO algorithms. Furthermore, in each iteration, our algorithm only requires $ \Theta(1) $ communication rounds and computational complexity, while other algorithms, such as those in \cite{chen2023decentralized,lu2022stochastic}, involve computation of $ \Theta(\log K) $ matrix-vector multiplications {per iteration}. This makes our algorithm computationally efficient and well-suited for large-scale DBO problems.

%\mbox{}\\
\noindent {\bf Roadmap of the Proof. }
Here we briefly describe the roadmap of the proof of Theorem \ref{cr} and the details are postponed to the Appendix.
First, we define the Lyapunov function
\begin{align}\label{Lyapunov}
V_k := & F(\bar{x}^k,y^*(\bar{x}^k)) - F^*
+ a_1 \|\bar{y}^k-y^*(\bar{x}^k)\|^2
+  a_2 \|\bar{v}^k-v^*(\bar{x}^k)\|^2 \nonumber\\
& + \frac{a_3}{n}\sum_{i=1}^{n}\|x_i^k-\x^k\|^2 + \frac{a_4}{n} \sum_{i=1}^{n}\|y_i^k-\y^k\|^2 + \frac{a_5}{n}\sum_{i=1}^{n}\|v_i^k-\v^k\|^2 \nonumber\\
& + \frac{a_6 \alpha^2}{n}\sum_{i=1}^{n}\|t_{x,i}^k-\t_x^k\|^2 + \frac{a_7 \beta^2}{n}\sum_{i=1}^{n}\|t_{y,i}^k-\t_y^k\|^2 + \frac{a_8 \eta^2}{n}\sum_{i=1}^{n}\|t_{v,i}^k-\t_v^k\|^2,
\end{align}
where we define $v^*(x) := \left[\nabla^2_{22} f(x,y^*(x))\right]^{-1}\nabla_2 F(x,y^*(x))$ and $ a_1,a_2,\ldots,a_8 $ are constants. We will show that, for another set of positive constants $A_1,A_2,\ldots,A_9$,
the following inequality holds:
\begin{align*}
V_{k+1} &- V_k \leq -\frac{\alpha}{2}\|\nabla \Phi (\bar{x}^k)\|^2
-\frac{A_1}{\alpha^2}\|\x^{k+1}-\x^k\|^2 - A_2 \|\y^k - y^*(\x^k)\|^2 -A_3  \|\v^k - v^*(\x^k)\|^2\\
&-  \frac{A_4}{n} \sum_{i=1}^n \|x_i^k - \x^k\|^2 - \frac{A_5 }{n} \sum_{i=1}^n \|y_i^k - \y^k\|^2 - \frac{A_6}{n} \sum_{i=1}^n \|v_i^k - \v^k\|^2\\
& -  \frac{A_7}{n} \sum_{i=1}^n \|t_{x,i}^k - \t_x^k\|^2 -  \frac{A_8}{n} \sum_{i=1}^n \|t_{y,i}^k - \t_y^k\|^2 -  \frac{A_9}{n} \sum_{i=1}^n \|t_{v,i}^k - \t_v^k\|^2.
\end{align*}
By taking the telescoping sum of this inequality, we can complete the proof.
The construction of the Lyapunov function \eqref{Lyapunov} and part of our proof were inspired by \cite{liu2023averaged}.

\section{Numerical Experiments}
In this section, we conduct experiments on hyperparameter optimization to evaluate the effectiveness of our proposed SLDBO. To test the decentralized setting, we used a local device equipped with 8 cores, i.e., $n = 8$, and employed mpi4py \cite{dalcin2021mpi4py} for parallel computing.

We adopted a ring topology to model the network for distributed computation, represented by a weight matrix
$W = (w_{ij})\in \mathbb{R}^{n\times n}$ given by: for $i,j=1,\ldots,n$,
$w_{ij}=w$ if $i=j$, $w_{ij}=(1-w)/2$ if $i=j\pm1$ or $(i,j)\in\{(1,n),(n,1)\}$, and $w_{ij}=0$ otherwise, where $w \in (0,1)$ is a constant. We take $w=0.4$ in our experiments.
In this ring topology, each agent has exactly two neighbours.
Our experiments involve both synthetic and real-world data.

\subsection{Synthetic Data}
We conduct logistic regression with $\ell_2$ regularization.
Let $\psi(t) = \log (1+e^{-t})$ for $t\in\mathbb{R}$ and $p$ be the dimension of the data.
Following \cite{chen2023decentralized}, on agent $i$, $i=1,\ldots,8$, we have
\begin{align*}
& F_i(\lambda, \omega)=\sum_{\left(x_e, y_e\right) \in \mathcal{D}_i^{\prime}} \psi\left(y_e x_e^{\top} \omega\right), \\
& f_i(\lambda, \omega)=\sum_{\left(x_e, y_e\right) \in \mathcal{D}_i} \psi\left(y_e x_e^{\top} \omega\right)+\frac{1}{2} \sum_{j=1}^p e^{\lambda_j} \omega_j^2,
\end{align*}
where $\mathcal{D}_i$ and $\mathcal{D}_i^{\prime}$ denote the training and  testing datasets on agent $i$, respectively.
We aim to identify the optimal hyperparameter $\lambda$ such that $\omega^*(\lambda)$ represents the optimal model parameter corresponding to $\lambda$.
To achieve this, we utilize synthetic heterogeneous data, generated in the same manner as in \cite{chen2023decentralized}. Specifically, the data distribution of $x_e$ on agent $i$ follows a normal distribution with mean $0$ and variance $i^2 \cdot r^2$, where $r$ is the heterogeneity rate. For the response variable, we let $y_e=x_e^{\top} w+0.1 z$, where $z$ is sampled from the standard normal distribution.

Note that the parameter $r$ controls the data heterogeneity rate. We first compare SLDBO with MA-DSBO (Moving Average-DSBO) \cite{chen2023decentralized} under conditions of low data heterogeneity with $r=1$ to investigate the advantages of the single-loop structure of SLDBO.
Full gradients are calculated in both algorithms, and we use a training dataset and a testing dataset consisting of 20,000 samples.
In SLDBO, we set $r_v=2$, ${\alpha}={\eta}=0.025$ and ${\beta} = 0.06$.
The results under different data dimensions are shown in Figures \ref{fig:d50}-\ref{fig:d200}.
MA-DSBO employs two key parameters, where $T$ represents the number of iterations performed in the inner loop, and $N$ represents the number of Hessian-inverse-gradient product iterations.
Comparing SLDBO and MA-DSBO, it is evident that SLDBO is faster. In particular,
SLDBO is much faster than MA-DSBO with $T=N=2$.
Note that MA-DSBO requires sufficient inner-loop iterations to accurately estimate the LL solution and the hypergradient.
By comparing the results in Figure \ref{fig:d50} ($p=50$) with those in Figure \ref{fig:d200} ($p=200$), we can observe that SLDBO demonstrates a greater improvement in convergence rate as the dimension of the data increases. SLDBO's single-loop structure offers an advantage in terms of reducing the number of required matrix-vector products, which proves to be particularly beneficial when dealing with high-dimensional data.

\begin{figure}[h]
	\centering
	\includegraphics[width=1\linewidth]{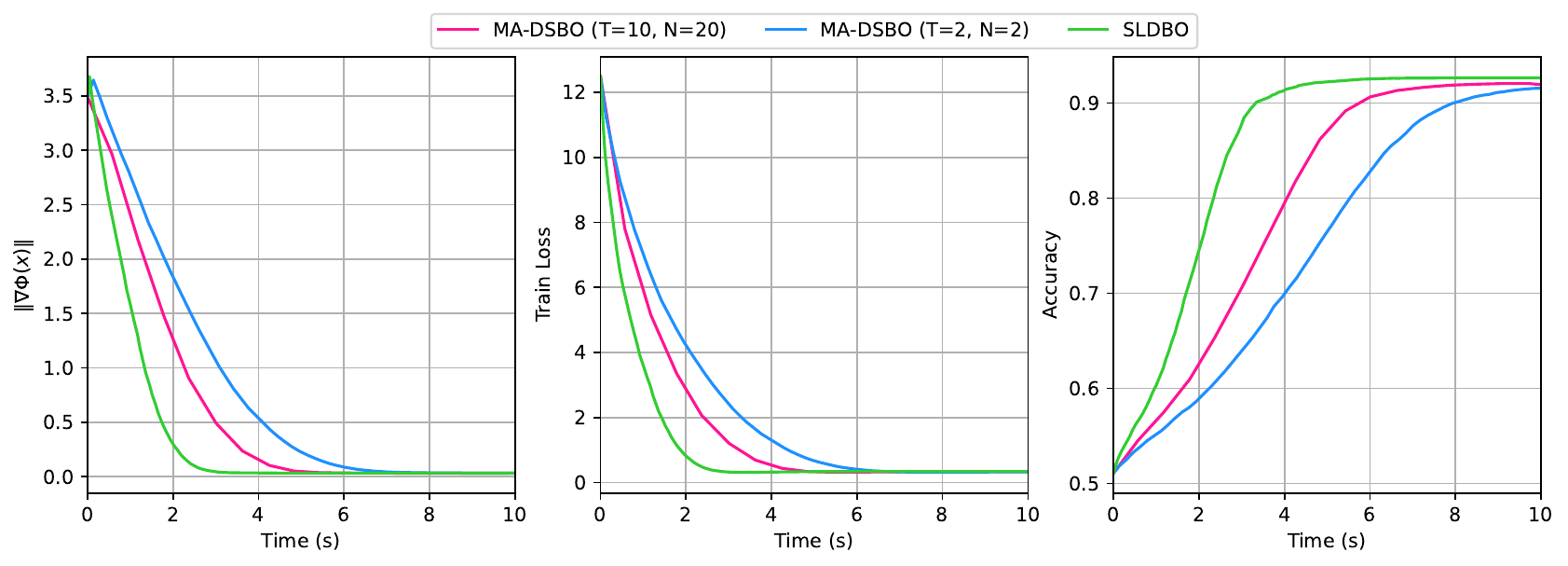}
	\caption{Comparison between MA-DSBO and SLDBO on synthetic data ($p=50$).}
	\label{fig:d50}
\end{figure}

\begin{figure}[h]
	\centering
	\includegraphics[width=1\linewidth]{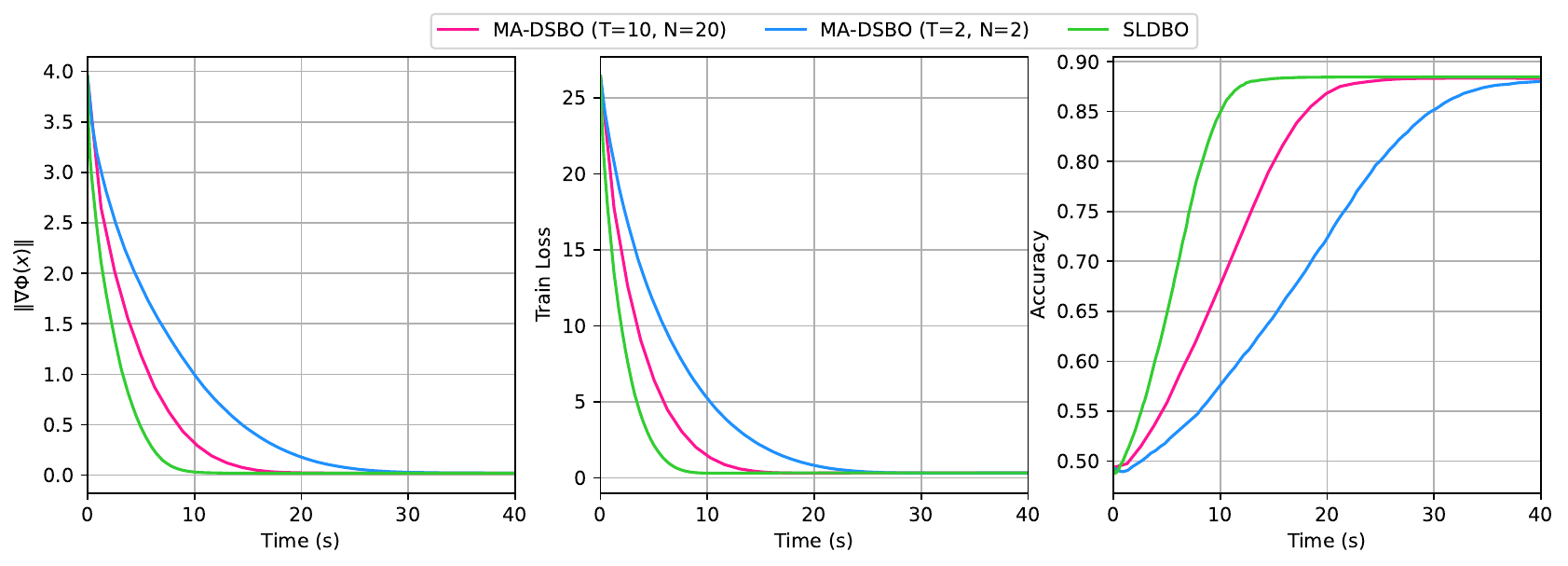}
	\caption{Comparison between MA-DSBO and SLDBO on synthetic data ($p=200$).}
	\label{fig:d200}
\end{figure}

Subsequently, we set different heterogeneity rates to highlight the exceptional performance of SLDBO under conditions of high data heterogeneity, as well as the necessity of the projection technique in SLDBO. The experimental results are illustrated in Figure \ref{fig:proj}. SLDBO (w/o proj.) represents SLDBO without the projection step, this means replacing \eqref{vupdate} with
\begin{align*}
t_{v, i}^k=\sum\nolimits_{j=1}^n w_{i j} t_{v, j}^{k-1}+d_{v,i}^k-d_{v, i}^{k-1},\quad v_i^{k+1}=\sum\nolimits_{j=1}^n w_{i j}( v_j^k+\eta  t_{v, j}^k).
\end{align*}
As shown in the left plot of Figure \ref{fig:proj}, when data heterogeneity is low ($ r=1 $), SLDBO with or without projection performs similarly, indicating that the projection step is almost ineffective. However, it is noteworthy that at a high level of data heterogeneity ($ r=40 $, the right plot of Figure \ref{fig:proj}), SLDBO without projection and MA-DSBO only achieve an accuracy of 0.5, while SLDBO (with projection) achieves an accuracy of about 0.96. This demonstrates the critical role of the projection technique in our proposed algorithm.

\begin{figure}[h]
	\centering
	\includegraphics[width=0.9\linewidth]{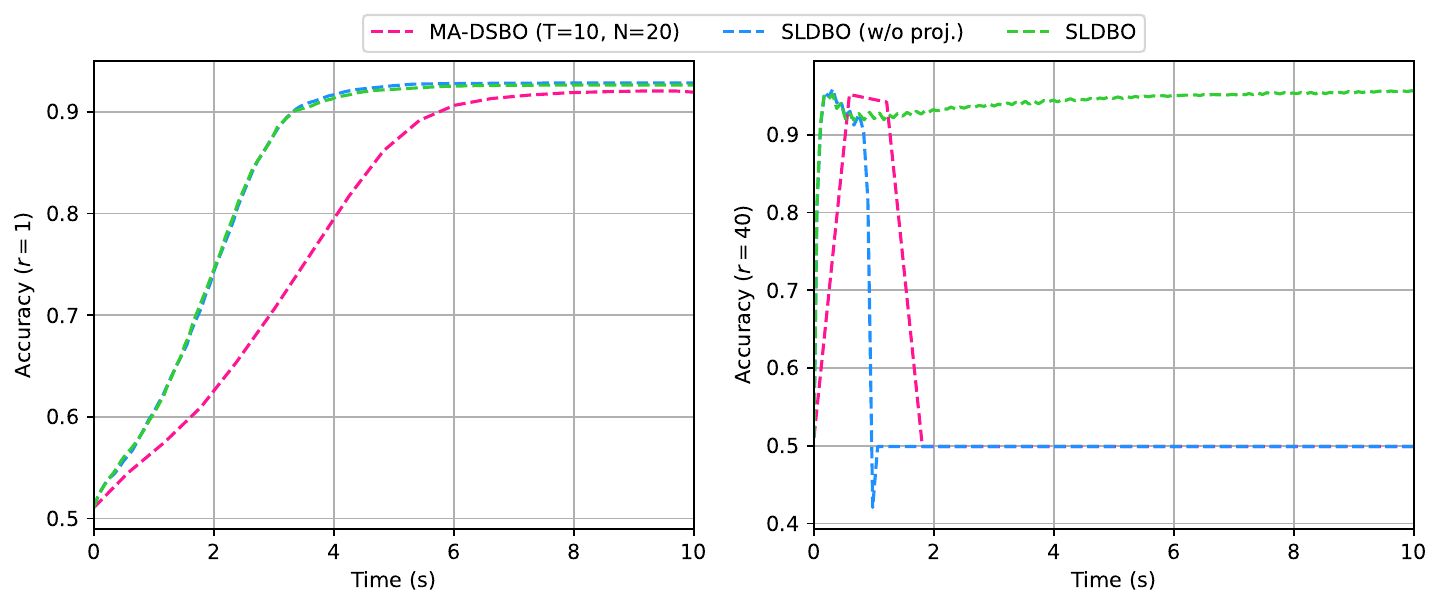}
	\caption{Comparison between MA-DSBO, SLDBO (w/o proj.) and SLDBO on synthetic data. Dimension: $ p=50 $. Heterogeneity rate: $ r=1 $ (left), $ r=40 $ (right).}
	\label{fig:proj}
\end{figure}

\subsection{Real-World Data}
Similar to the case of synthetic data, we define $\mathcal{D}_i$ and $\mathcal{D}_i^{\prime}$ as the training and testing datasets, respectively, for agent $i$.
We next apply SLDBO to solve the following hyperparameter problem using the MNIST database \cite{lecun1998gradient}:
\begin{align*}
F_i(\lambda, \omega)&= \frac{1}{|\mathcal{D}_i^{\prime}|}\sum_{(x_e,y_e)\in\mathcal{D}_i^{\prime}} L(x_e^{\top}\omega, y_e), \\
f_i(\lambda, \omega)&= \frac{1}{|\mathcal{D}_i|}\sum_{(x_e,y_e)\in \mathcal{D}_i}L(x_e^{\top}\omega, y_e) + \frac{1}{cp}\sum_{i=1}^{c}\sum_{j=1}^{p}e^{\lambda_j}\omega_{ij}^2,
\end{align*}
where $\omega \in \mathbb{R}^{c\times p}$ denotes the model parameter, $L$ denotes the cross entropy loss, and $|S|$ denotes the cardinality of a set $S$.
In our experiments, we set $c$ and $p$ to be $10$ and $784$, respectively. Here $c$ represents the number of classes and $p$ represents the number of features.
The training and testing sets comprise 60,000 samples each, with balanced representation across all classes.
To reduce the computational overhead associated with estimating gradients from a large dataset in our SLDBO algorithm, we adopt a technique inspired by stochastic gradient descent. Specifically, we extract a representative subset of samples to estimate the gradients instead of using the entire dataset.
For both the SLDBO and MA-DSBO algorithms, we set the batch size on each computing agent to 1,000. In the case of SLDBO, the hyperparameters were set as follows: $r_v=10$, ${\alpha}={\eta}=0.024$, and ${\beta}=0.06$. For MA-DSBO, we set $T=N=5$. Figure \ref{fig:slmnist} presents a comparison of the test loss, train loss, and classification accuracy between the SLDBO and MA-DSBO algorithms. These results demonstrate that our proposed algorithm SLDBO can efficiently solve this problem with improved convergence rate and classification accuracy.
\begin{figure}[h]
	\centering
	\includegraphics[width=0.9\linewidth]{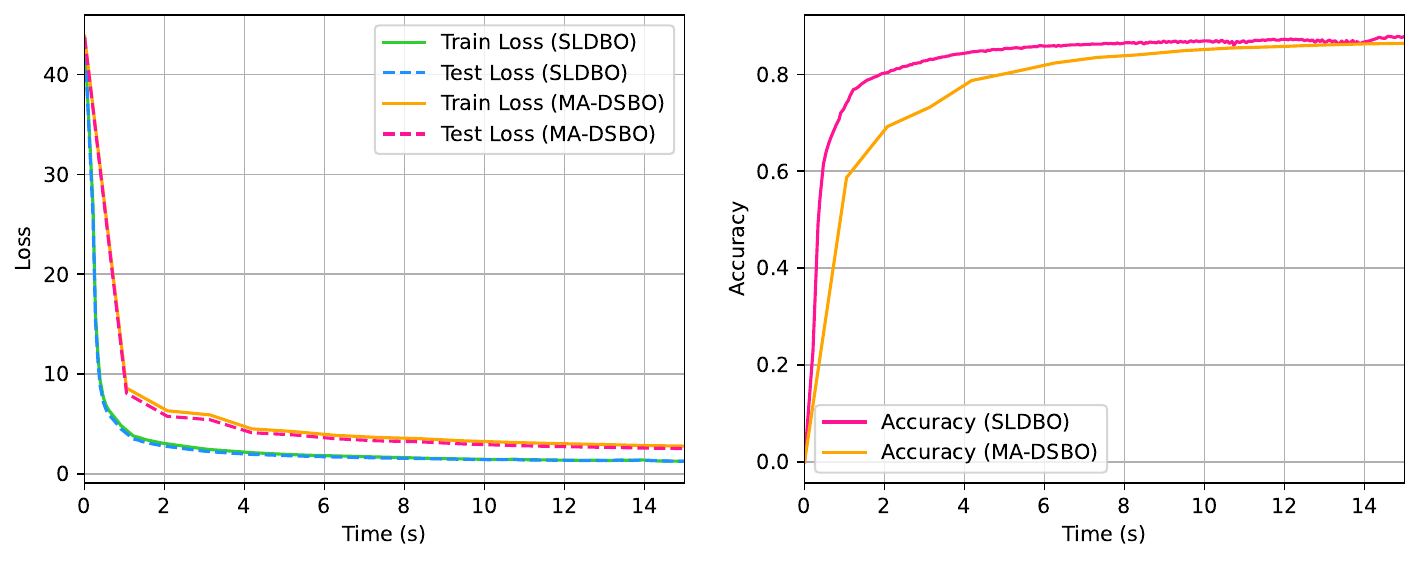}
	\caption{Comparison of test loss, train loss, and classification accuracy between MA-DSBO and SLDBO on real-world MNIST dataset.}
	\label{fig:slmnist}
\end{figure}

\section{Concluding Remarks}
This paper presents a novel single-loop algorithm, called SLDBO, for efficiently solving DBO problems with a guaranteed sublinear convergence rate. Notably, SLDBO is the first single-loop algorithm for DBO that operates without nested matrix-vector products (only two matrix-vector products are required at each iteration) and does not make any assumptions related to data  heterogeneity. Our numerical experiments confirm the effectiveness of SLDBO. 
Nevertheless, we acknowledge that computing the full gradient during practical applications can be time-consuming. Therefore, it is worth investigating the extension of our algorithm to the stochastic setting. Additionally, techniques for saving communications, such as those used in E-AiPOD \cite{xiao2023alternating}, could be integrated into our algorithm for decentralized settings. These topics are left for further investigation.	

\begin{spacing}{1}
\small
\bibliographystyle{abbrv}
\bibliography{sldbo.bib}
\end{spacing}

\appendix

\section{Proof of the Convergence Results}\label{appendix}

In this appendix, we provide the proof of our convergence results, i.e., Theorem \ref{cr}. Assumption \ref{Assump1} and \ref{Assump2} are used throughout the proof.

\subsection{Notation, Constants and Basic Lemmas}

For the ease of presentation, we define some notation below.
\begin{equation}\label{def:bar9}
\left\{
\begin{array}{lll}
{\x^{k}} := \frac{1}{n}\sum_{i=1}^{n}x_{i}^k, & {\y^{k}} := \frac{1}{n}\sum_{i=1}^{n}y_{i}^k, & {\v^{k}} := \frac{1}{n}\sum_{i=1}^{n}v_{i}^k, \smallskip\\
\d_x^k := \frac{1}{n}\sum_{i=1}^{n}d_{x,i}^k, & \d_y^k := \frac{1}{n}\sum_{i=1}^{n}d_{y,i}^k, & \d_v^k := \frac{1}{n}\sum_{i=1}^{n}d_{v,i}^k, \smallskip\\
\t_x^k := \frac{1}{n}\sum_{i=1}^{n}t_{x,i}^k, & \t_y^k := \frac{1}{n}\sum_{i=1}^{n}t_{y,i}^k, & \t_v^k := \frac{1}{n}\sum_{i=1}^{n}t_{v,i}^k.
\end{array}
\right.
\end{equation}
Recall that $r_v$ is defined in  \eqref{def-constants-main-text}.
To assist in our proof, we now define some additional constants as follows:
\begin{equation} \label{def-constants-proof}
C_1  := 3\max\{(L_{F,1}+r_{v}L_{f,2})^2,L_{f,1}^2\}, \quad C_2  := L_{f,1}^2.
\end{equation}

Before the proof, we present some useful lemmas. Among others, Lemma \ref{con_sqr} is a common result in decentralized optimization, see, e.g., \cite[Lemma 1]{pu2021distributed}. For completeness, we provide the proof.

\begin{lemma}\label{con_sqr}
	Consider the mixing matrix $W=(w_{ij})\in \mathbb{R}^{n\times n}$ defined in Assumption \ref{Assump2}, for any $x_1, \ldots, x_n \in \mathbb{R}^{d}$, let $\x = \frac{1}{n}\sum_{i=1}^{n}x_i$, we have
	\begin{description}
		\item[(a)] $	\sum_{i=1}^n\big\|\sum_{j=1}^n w_{i j}x_j\big\|^2\leq  \sum_{j=1}^n\left\|x_j\right\|^2$.
		\item[(b)] $\sum_{i=1}^n\big\|\sum_{j=1}^n w_{i j}\left(x_j-\x\right)\big\|^2\leq \rho^2 \sum_{i=1}^n\left\|x_i-\x\right\|^2$.
	\end{description}
\end{lemma}
\begin{proof}
{\bf (a).} $\sum_{i=1}^n \|\sum_{j=1}^n w_{i j}x_j\|^2\leq \sum_{i=1}^n\sum_{j=1}^n w_{i j}\|x_j\|^2 = \sum_{j=1}^n\sum_{i=1}^n w_{i j}\|x_j\|^2 = \sum_{j=1}^n\|x_j\|^2$,
where the inequality follows from the convexity of $\|\cdot\|^2$ and the last equality is due to $\sum_{i=1}^n w_{i j}=1$.
{\bf (b).} Define $X = [x_1^{\top}; x_2^{\top}; \ldots; x_n^{\top}]\in \mathbb{R}^{n\times d}$.
Then,        we have
\begin{align}\label{w1}
\left\|WX-\frac{1}{n}\mathbf{1_n}\mathbf{1_n^{\top}}X \right\|  \leq \left\|W-\frac{1}{n}\mathbf{1_n}\mathbf{1_n^{\top}} \right\|_{\textrm{op}} \left\|X-\frac{1}{n}\mathbf{1_n}\mathbf{1_n^{\top}}X\right\|,
\end{align}
where the inequality holds because $\left(W-\frac{1}{n}\mathbf{1_n}\mathbf{1_n^{\top}}\right)\frac{1}{n}\mathbf{1_n}\mathbf{1_n^{\top}}X=0$ by Assumption \ref{Assump2}.
Since $W$ and $\frac{1}{n}\mathbf{1_n}\mathbf{1_n^{\top}}$ are both symmetric and commute with each other, they are simultaneously diagonalizable, i.e., there exists an orthonormal matrix $P$ such that
$$W = P\text{diag}(\lambda_1, \lambda_2, \ldots, \lambda_n)P^{-1} \text{~~and~~} \frac{1}{n}\mathbf{1_n}\mathbf{1_n^{\top}} =  P\text{diag}(1, 0,\cdots,0)P^{-1},$$
where the eigenvalues of $W$ satisfy $1 = \lambda_1> \lambda_2\geq \ldots\geq \lambda_n$ and $\rho = \max\{|\lambda_2|, |\lambda_n|\} < 1$, as stated in Assumption \ref{Assump2}. Hence, it is easy to verify that
$\|W-\frac{1}{n}\mathbf{1_n}\mathbf{1_n^{\top}}\|_{\textrm{op}} = \|\text{diag}(0, \lambda_2, \ldots, \lambda_n)\|_{\textrm{op}} = \rho$.
Moreover, we have
\begin{align}\label{w3}
\left\|WX-\frac{1}{n}\mathbf{1_n}\mathbf{1_n^{\top}}X \right\|^2 = \left\|WX-\mathbf{1_n}\x^{\top}\right\|^2=\sum\nolimits_{i=1}^n\left\|\sum\nolimits_{j=1}^n w_{i j}\left(x_j-\x\right)\right\|^2
\end{align}
and
\begin{align}\label{w4}
\left\|X-\frac{1}{n}\mathbf{1_n}\mathbf{1_n^{\top}}X\right\|^2 = \left\|X-\mathbf{1_n}\x^{\top}\right\|^2=\sum_{i=1}^n\left\|x_i-\x\right\|^2.
\end{align}
The desired result follows by combining \eqref{w1}-\eqref{w4}
and $\|W-\frac{1}{n}\mathbf{1_n}\mathbf{1_n^{\top}}\|_{\textrm{op}} = \rho$.
\end{proof}

Lemma \ref{proj} is adopted from Lemma 3.2 in \cite{choi2023convergence}, and it describes the relationship between consensus error before and after projection. Recall that $\mathcal{P}_{r}[z]:=\argmin_{\{z':\|z'\|\in[0,r]\}}\|z'-z\|$.

\begin{lemma}\label{proj}
	For any $x_1, \ldots, x_n \in \mathbb{R}^{d}$, we have
	\begin{align*}    \sum_{i=1}^n\left\|\mathcal{P}_r\left[x_i\right]-\frac{1}{n} \sum\nolimits_{j=1}^n \mathcal{P}_r\left[x_j\right]\right\|^2 \leq \sum_{i=1}^n\left\|x_i-\frac{1}{n} \sum\nolimits_{j=1}^n x_j\right\|^2.
	\end{align*}
\end{lemma}

\begin{frame}
Consider the function $G(x) = \sum_{i=1}^n\left\|\mathcal{P}_r\left[x_i\right]-x\right\|^2$,
which is minimized by $x = \frac{1}{n}\sum_{i=1}^{n}\mathcal{P}_r\left[x_i\right]$. Therefore, we have
$$\sum_{i=1}^n\left\|\mathcal{P}_r\left[x_i\right]-\frac{1}{n} \sum\nolimits_{j=1}^n \mathcal{P}_r\left[x_j\right]\right\|^2 \leq \sum_{i=1}^n\left\|\mathcal{P}_r\left[x_i\right]-\mathcal{P}_r\left[\frac{1}{n} \sum\nolimits_{j=1}^n x_j\right]\right\|^2\leq \sum_{i=1}^n\left\|x_i-\frac{1}{n} \sum\nolimits_{j=1}^n x_j\right\|^2,$$
where the last inequality holds because the projection operator is non-expansive.
\end{frame}

\subsection{Consensus Error of Algorithm \ref{alg:slDB}}

Recall that ${\x^{k}}$, ${\y^{k}}$, ${\v^{k}}$, $\d_x^k$, $\d_y^k$, $\d_v^k$, $\t_x^k$, $\t_y^k$ and $\t_v^k$ are defined in \eqref{def:bar9}. In this section, we bound the terms related to the consensus error in \eqref{Lyapunov}.
We first prove some useful lemmas.

\begin{lemma}\label{bar}
	The sequence $\{v_i^k\}$ generated by Algorithm \ref{alg:slDB} satisfies
	\begin{align}\label{v-vbar}
	\sum_{i=1}^{n}\|v_i^{k+1}-\v^{k+1}\|^2 \leq \rho\sum_{i=1}^{n}\|v_i^{k} - \v^k\|^2 + \frac{\rho^2\eta ^2}{1-\rho}\sum_{i=1}^{n}\|t_{v,i}^k - \t_v^k\|^2.
	\end{align}
\end{lemma}
\begin{proof}
	In fact, we have
	\begin{align*}
	\sum_{i=1}^{n}\|v_i^{k+1}&-\v^{k+1}\|^2
	= \sum_{i=1}^{n} \left\|\mathcal{P}_{r_v}\left[\sum\nolimits_{j=1}^{n}w_{ij}(v_j^{k}+\eta t_{v,j}^k)\right]-\frac{1}{n}\sum_{s=1}^{n}\mathcal{P}_{r_v}\left[\sum\nolimits_{j=1}^{n}w_{sj}(v_j^{k}+\eta t_{v,j}^k)\right]\right\|^2\\
	& \leq   \sum_{i=1}^{n} \left\|\sum_{j=1}^{n}w_{ij}(v_j^{k}+\eta t_{v,j}^k)-\frac{1}{n}\sum_{s=1}^{n}\sum_{j=1}^{n}w_{sj}(v_j^{k}+\eta t_{v,j}^k)\right\|^2\\
	&\leq \left(1+\frac{1-\rho}{\rho}\right) \sum_{i=1}^{n}\left\|\sum\nolimits_{j=1}^{n}w_{ij}v_j^{k} - \v^k\right\|^2 + \left(1+\frac{\rho}{1-\rho}\right)\eta^2\sum_{i=1}^{n}\left\|\sum\nolimits_{j=1}^{n}w_{ij}t_{v,j}^k - \t_v^k\right\|^2\\
	&\leq  \rho\sum_{i=1}^{n}\|v_i^{k} - \v^k\|^2 + \frac{\rho^2\eta ^2}{1-\rho}\sum_{i=1}^{n}\|t_{v,i}^k - \t_v^k\|^2,
	\end{align*}
	where the first inequality follows from Lemma \ref{proj},
	the second follows from Cauchy-Schwartz inequality and the notation defined in \eqref{def:bar9}, 
	the third follows from Lemma \ref{con_sqr}.
\end{proof}

\begin{remark}
	By following similar deductions as in Lemma \ref{bar}, we can derive (details are omitted to avoid redundancy)
	\begin{align}
	\sum_{i=1}^{n}\|x_i^{k+1}-\x^{k+1}\|^2 \leq \rho\sum_{i=1}^{n}\|x_i^{k} - \x^k\|^2 + \frac{\rho^2\alpha ^2}{1-\rho}\sum_{i=1}^{n}\|t_{x,i}^k - \t_x^k\|^2, \label{x-xbar}\\
	\sum_{i=1}^{n}\|y_i^{k+1}-\y^{k+1}\|^2 \leq \rho\sum_{i=1}^{n}\|y_i^{k} - \y^k\|^2 + \frac{\rho^2\beta ^2}{1-\rho}\sum_{i=1}^{n}\|t_{y,i}^k - \t_y^k\|^2. \label{y-ybar}
	\end{align}
\end{remark}

\begin{lemma}\label{tx_bounded}
	The sequences $\{t_{x,i}^k\}$ and $\{d_{x,i}^k\}$ generated by Algorithm \ref{alg:slDB} satisfy
	\begin{align}
	\sum_{i=1}^{n}\|t_{x,i}^{k+1}-\t_x^{k+1}\|^2 \leq \rho \sum_{i=1}^{n}\|t_{x,i}^{k}-\t_x^{k}\|^2 + \frac{1}{1-\rho}\sum_{i=1}^{n}\|d_{x,i}^{k+1} - d_{x,i}^{k}\|^2. \label{tx-bartx}
	\end{align}
\end{lemma}
\begin{proof}
    From the update of $x_i^k$ and $t_{x,i}^k$ in \eqref{xupdate}, we have
$\t^k_x=\t^{k-1}_x+\d^k_x-\d^{k-1}_x$, 	$\t^{-1}_x=\d^{-1}_x=0 $ and $\x^{k+1}=\x^{k}-\alpha\t^k_x$,
	which implies (by induction)
	\begin{align} \label{xiter}
	\t^k_x = \d^k_x \text{~~and~~} \x^{k+1}=\x^{k}-\alpha\d^k_x.
    \end{align}
    Therefore, we have
    \begin{align*}	&\sum_{i=1}^{n}\|t_{x,i}^{k+1}-\t_x^{k+1}\|^2
    = \sum_{i=1}^{n}\left\|\left(\sum\nolimits_{j=1}^{n}w_{ij}t_{x,j}^{k} + d_{x,i}^{k+1} - d_{x,i}^{k}- \t_x^{k}\right) + \big(\t_x^{k} - \t_x^{k+1}\big)\right\|^2\\
	=& \sum_{i=1}^{n}\left\|\sum\nolimits_{j=1}^{n}w_{ij}t_{x,j}^{k} + d_{x,i}^{k+1} - d_{x,i}^{k}- \t_x^{k}\right\|^2
 + 2\sum_{i=1}^{n}\left\langle
 \t_{x,i}^{k+1}- \t_x^{k}, \t_x^{k} - \t_x^{k+1}\right\rangle + \sum_{i=1}^{n}\|\t_x^{k} - \t_x^{k+1}\|^2   \\
	=&\sum_{i=1}^{n}\left\|\sum\nolimits_{j=1}^{n}w_{ij}t_{x,j}^{k} + d_{x,i}^{k+1} - d_{x,i}^{k}- \t_x^{k}\right\|^2 - n\|\d_x^{k} - \d_x^{k+1}\|^2\\
	\leq&\left(1+\frac{1-\rho}{\rho}\right)\sum_{i=1}^{n}\left\|\sum\nolimits_{j=1}^{n}w_{ij}t_{x,j}^{k} - \t_x^{k}\right\|^2 + \left(1+\frac{\rho}{1-\rho}\right)\sum_{i=1}^{n}\|d_{x,i}^{k+1} - d_{x,i}^{k}\|^2 - n\|\d_x^{k} - \d_x^{k+1}\|^2\\
	\leq & \, \rho \sum_{i=1}^{n}\|t_{x,i}^{k}-\t_x^{k}\|^2 + \frac{1}{1-\rho}\sum_{i=1}^{n}\|d_{x,i}^{k+1} - d_{x,i}^{k}\|^2,
	\end{align*}
	where the first inequality is due to Cauchy-Schwartz inequality, and the second inequality follows from Lemma \ref{con_sqr}. 
\end{proof}

\begin{remark}\label{tv_bound}
The following inequalities can be derived by following similar steps as in Lemma \ref{tx_bounded}: 
\begin{align}
\sum_{i=1}^{n}\|t_{v,i}^{k+1}-\t_v^{k+1}\|^2 &\leq \rho \sum_{i=1}^{n}\|t_{v,i}^{k}-\t_v^{k}\|^2 + \frac{1}{1-\rho}\sum_{i=1}^{n}\|d_{v,i}^{k+1} - d_{v,i}^{k}\|^2, \label{tv-bartv}\\
\sum_{i=1}^{n}\|t_{y,i}^{k+1}-\t_y^{k+1}\|^2 &\leq \rho \sum_{i=1}^{n}\|t_{y,i}^{k}-\t_y^{k}\|^2 + \frac{1}{1-\rho}\sum_{i=1}^{n}\|d_{y,i}^{k+1} - d_{y,i}^{k}\|^2. \label{ty-barty}
\end{align}
The details are omitted.
\end{remark}

\begin{lemma}\label{dk+1-dk}
The sequences  $\{d_{x,i}^{k}\}$, $\{d_{v,i}^{k}\}$ and $\{d_{y,i}^{k}\}$ generated by Algorithm \ref{alg:slDB} satisfy
	\begin{align}
	\sum_{i=1}^{n}\|d_{x,i}^{k}-d_{x,i}^{k-1}\|^2&\leq C_1 \sum_{i=1}^{n}\left(\|x_i^k-x_i^{k-1}\|^2+\|y_i^k-y_i^{k-1}\|^2+\|v_i^k-v_i^{k-1}\|^2\right),\label{dx}\\
	\sum_{i=1}^{n}\|d_{v,i}^{k}-d_{v,i}^{k-1}\|^2&\leq C_1 \sum_{i=1}^{n}\left(\|x_i^k-x_i^{k-1}\|^2+\|y_i^k-y_i^{k-1}\|^2+\|v_i^k-v_i^{k-1}\|^2\right),\label{dv}\\
	\sum_{i=1}^{n}\|d_{y,i}^{k}-d_{y,i}^{k-1}\|^2&\leq C_2  \sum_{i=1}^{n}\left(\|x_i^k-x_i^{k-1}\|^2+\|y_i^k-y_i^{k-1}\|^2\right)\label{dy},
	\end{align}
	where $C_1$ and $C_2$ are defined in \eqref{def-constants-proof}.
\end{lemma}
\begin{proof}
	By the definition of $d_{x,i}^k$ in \eqref{xd}, it follows that
	\begin{align*}
	&\sum_{i=1}^{n}\|d_{x,i}^{k}-d_{x,i}^{k-1}\|^2
%	=  \sum_{i=1}^{n} \|\nabla_1 F(x^k_i, y^k_i) - \nabla^2_{12} f_i (x^k_i, y^k_i) v^k_i - \nabla_1 F(x^{k-1}_i, y^{k-1}_i) + \nabla^2_{12} f_i (x^{k-1}_i, y^{k-1}_i) v^{k-1}_i\|^2\\
	 \leq 3\sum_{i=1}^{n}\|\nabla_1 F_i(x^k_i, y^k_i)-\nabla_1F_i(x^{k-1}_i, y^{k-1}_i)\|^2\\
	 &\ \ \quad  +3\sum_{i=1}^{n}\|(\nabla^2_{12} f_i (x^k_i, y^k_i) -\nabla^2_{12} f_i (x^{k-1}_i, y^{k-1}_i)) v^{k}_i\|^2 +3\sum_{i=1}^{n}\|\nabla^2_{12} f_i (x^{k-1}_i, y^{k-1}_i)(v^{k}_i-v^{k-1}_i)\|^2\\
	&\leq 3(L_{F,1}+r_{v}L_{f,2})^2\sum_{i=1}^{n}\|x^k_i-x^{k-1}_i\|^2 + 3(L_{F,1}+r_{v}L_{f,2})^2\sum_{i=1}^{n}\|y^k_i-y^{k-1}_i\|^2+3L_{f,1}^2\sum_{i=1}^{n}\|v^k_i-v^{k-1}_i\|^2\\
	 &\leq C_1 \sum_{i=1}^{n}\left(\|x_i^k-x_i^{k-1}\|^2+\|y_i^k-y_i^{k-1}\|^2+\|v_i^k-v_i^{k-1}\|^2\right).
	\end{align*}
 Similarly, by the definition of $d_{v,i}^k$ in \eqref{vd}, we can also derive
	\begin{align*}
	\sum_{i=1}^{n}\|d_{v,i}^{k}-d_{v,i}^{k-1}\|^2\leq C_1 \sum_{i=1}^{n}\left(\|x_i^k-x_i^{k-1}\|^2+\|y_i^k-y_i^{k-1}\|^2+\|v_i^k-v_i^{k-1}\|^2\right).
	\end{align*}
	Then by the definition of $d_{y,i}^k$ in \eqref{yd}, we obtain
	\begin{align*}
	&\sum_{i=1}^{n}\|d_{y,i}^{k}-d_{y,i}^{k-1}\|^2\leq \sum_{i=1}^{n}\|\nabla_2f(x_i^k,y_i^{k})-\nabla_2f(x_i^{k-1},y_i^{k-1})\|^2\\
	&\ \ \leq L_{f,1}^2\sum_{i=1}^{n}\|x_i^k-x_i^{k-1}\|^2 + L_{f,1}^2\sum_{i=1}^{n}\|y_i^k-y_i^{k-1}\|^2
	= C_2  \sum_{i=1}^{n}\left(\|x_i^k-x_i^{k-1}\|^2+\|y_i^k-y_i^{k-1}\|^2\right),
	\end{align*}
which completes the proof.
\end{proof}

\begin{lemma}\label{k+1-k}
The sequence $\{v_i^k\}$ generated by Algorithm \ref{alg:slDB} satisfies
\begin{align}\label{vk+1-vk}
	\sum_{i=1}^{n} \|v_i^{k+1} - v_{i}^{k}\|^2 \leq 8 \sum_{i=1}^{n}\|v_i^{k}-\v^{k}\|^2 + 4\eta^2 \sum_{i=1}^{n}\|t_{v,i}^{k}-\t_v^{k}\|^2 + 4n\eta^2\|\d_v^{k}\|^2.
\end{align}
\end{lemma}
\begin{proof}
Follow the update of $ v_i^k $ in \eqref{vupdate}, we have
\begin{align*}
&\sum_{i=1}^{n} \|v_i^{k+1} - v_{i}^{k}\|^2 = \sum_{i=1}^{n} \| \mathcal{P}_{r_v}\left[\sum\nolimits_{j=1}^n w_{i j}( v_j^k+\eta  t_{v, j}^k)\right]- v_{i}^{k}\|^2
\leq \sum_{i=1}^{n} \| \sum\nolimits_{j=1}^n w_{i j}( v_j^k+\eta  t_{v, j}^k)- v_{i}^{k}\|^2 \\
&\leq  4\sum_{i=1}^{n} \| \sum\nolimits_{j=1}^n w_{i j} v_j^k-\v^k\|^2 + 4\sum_{i=1}^{n} \| \v^k - v_i^k\|^2 + 4\eta^2\sum_{i=1}^{n} \| \sum\nolimits_{j=1}^n w_{i j} t_{v,j}^k-\t_v^k\|^2 + 4\eta^2\sum_{i=1}^{n} \|\d_v^k\|^2\\
&\leq  8 \sum_{i=1}^{n}\|v_i^{k}-\v^{k}\|^2 + 4\eta^2 \sum_{i=1}^{n}\|t_{v,i}^{k}-\t_v^{k}\|^2 + 4\eta^2 \sum_{i=1}^{n}\|\d_v^{k}\|^2,
\end{align*}
where the first inequality holds because the projection operator is non-expansive, and the third inequality is due to Lemma \ref{con_sqr}.
\end{proof}

\begin{remark}
Similarily, we can deduce the following inequalities, details are omitted to avoid repetition.
\begin{align}
	\sum_{i=1}^{n} \|x_i^{k+1} - x_{i}^{k}\|^2 &\leq 8 \sum_{i=1}^{n}\|x_i^{k}-\x^{k}\|^2 + 4\alpha^2 \sum_{i=1}^{n}\|t_{x,i}^{k}-\t_x^{k}\|^2 + 4n\alpha^2\|\d_x^{k}\|^2, \label{xk+1-xk}\\
	\sum_{i=1}^{n} \|y_i^{k+1} - y_{i}^{k}\|^2 &\leq 8 \sum_{i=1}^{n}\|y_i^{k}-\y^{k}\|^2 + 4\beta^2 \sum_{i=1}^{n}\|t_{y,i}^{k}-\t_y^{k}\|^2 + 4n\beta^2\|\d_y^{k}\|^2. \label{yk+1-yk}
\end{align}
\end{remark}

\begin{lemma}
The sequences $\{d_{y,i}^k\},\,\{d_{v,i}^k\}$ generated by Algorithm \ref{alg:slDB} satisfy
\begin{align}
\|\d_y^{k}\|^2 &\leq \frac{2L_{f,1}^2}{n}\sum_{i=1}^{n}\left(\|x_i^k - \x^k\|^2 + \|y_i^k - \y^k\|^2\right) + 2L_{f,1}^2\|\y^k - y^*(\x^k)\|^2,\label{bardy}\\
\|\d_v^{k}\|^2 &\leq  \frac{5(L_{f,2}r_v+L_{F,1})^2}{n}\sum_{i=1}^n \left(\|x_i^k-\x^k\|^2+\|y_i^k-\y^k\|^2\right)
+ \frac{5L_{f,1}^2}{n}\sum_{i=1}^n \|v_i^k-\v^k\|^2 \nonumber\\
&\quad + 5 (L_{F,1}+r_v L_{f,2})^2 \|\y^k - y^*(\x^k)\|^2 + 5 L_{f,1}^2 \|\v^k - v^*(\x^k)\|^2.\label{bardv}
\end{align}
\end{lemma}
\begin{proof}
For $ \nabla_2 f(\x^k,y^*(\x^k))=0 $, by the definition of $ d_{y,i}^k $ in \eqref{yd}, we obtain
\begin{align*}
\|\d_y^{k}\|^2 &\leq \frac{2}{n}\sum_{i=1}^{n} \|\nabla_2 f_i(x_i^k,y_i^k) - \nabla_2 f_i(\x^k,\y^k)\|^2 + 2\|\nabla_2 f(\x^k,\y^k) - \nabla_2 f(\x^k,y^*(\x^k))\|^2\\
& \leq \frac{2L_{f,1}^2}{n}\sum_{i=1}^{n}\left(\|x_i^k - \x^k\|^2 + \|y_i^k - \y^k\|^2\right) + 2L_{f,1}^2\|\y^k - y^*(\x^k)\|^2.
\end{align*}
For convenience, we define
\begin{equation*}
\left\{
\begin{array}{l}
\Delta_1 := \left\|\nabla_2 F(\x^k,y^* (\x^k))-\nabla^2_{22}f (\x^k,y^* (\x^k)) v^* (\x^k)-\nabla_2 F(\x^k,\y^k)+\nabla^2_{22}f (\x^k,\y^k) \v^k\right\|, \smallskip \\
\Delta_2 := \left\|\nabla_2 F(\x^k,\y^k)-\nabla^2_{22}f(\x^k,\y^k) \v^k- \frac{1}{n}\sum_{i=1}^n\Big(\nabla_2 F_i(x^k_i, y^k_i) - \nabla^2_{22} f_i(x^k_i, y^k_i) v^k_i\Big)\right\|.
\end{array}
\right.
\end{equation*}
From the triangle inequality, it follows that
\begin{align}\label{dv1}
\Delta_1	\leq \, & \Big\|\nabla_2 F(\x^k, y^* (\x^k))-\nabla_2 F(\x^k,\y^{k})\Big\|
+\left\|\Big(\nabla^2_{22} f (\x^k,\y^k) - \nabla^2_{22} f (\x^k, y^* (\x^k))\Big)  \v^k\right\| \nonumber\\
\, & +\left\|\nabla^2_{22} f(\x^k, y^* (\x^k))\big(\v^{k}-v^* (\x^k)\big)\right\|\nonumber\\
\leq \, &
\left(L_{F,1} + L_{f,2} r_v\right) \|\y^k-y^* (\x^k)\|
+L_{f,1} \|\v^{k}-v^* (\x^k)\|.
\end{align}
By using the triangle inequality again and considering $F={1\over n}\sum_{i=1}^nF_i$ and $f={1\over n}\sum_{i=1}^nf_i$, we derive
\begin{align}\label{dv2}
\Delta_2	\leq \, & \frac{1}{n}\sum_{i=1}^n \left\|\nabla_2 F_i(\x^k,\y^k)-\nabla^2_{22}f_i(\x^k,\y^k) \v^k- \nabla_2 F_i(x^k_i, y^k_i) + \nabla^2_{22} f_i(x^k_i, y^k_i) v^k_i)\right\|\nonumber\\
\leq \, & \frac{1}{n} \sum_{i=1}^n\left\| \nabla_2 F_i(\x^k,\y^k)- \nabla_2 F_i(x^k_i, y^k_i)\right\| + \frac{1}{n} \sum_{i=1}^n\left\|\Big(\nabla^2_{22} f_i(\x^k,\y^k)-\nabla^2_{22} f_i(x_i^k, y_i^k)\Big)  v^k_i\right\| \nonumber\\
\, & + \frac{1}{n}\sum_{i=1}^n\Big\|\nabla^2_{22} f_i(\x^k,\y^{k})(\v^k-v_i^k)\Big\|\nonumber\\
\leq \, & \frac{L_{F,1}+ r_v L_{f,2}}{n}\sum_{i=1}^n \big(\|x_i^k-\x^k\|+ \|y_i^k-\y^k\|\big)
+ \frac{L_{f,1}}{n}\sum_{i=1}^n \|v_i^k-\v^k\|.
\end{align}
Then, for $ \nabla_2 F(\x^k,y^*(\x^k)) = \nabla^2_{22} f(\x^k,y^*(\x^k)) v^*(\x^k)$, combining \eqref{dv1} and \eqref{dv2}, using the inequality $\left(\sum_{l=1}^{s} a_l\right)^2\leq s\sum_{l=1}^{s} a_l^2$, we can deduce that
\begin{align*}
\|\d_v^{k}\|^2  
&\leq (\Delta_1 + \Delta_2)^2\leq  \frac{5(L_{f,2}r_v+L_{F,1})^2}{n}\sum_{i=1}^n \left(\|x_i^k-\x^k\|^2+\|y_i^k-\y^k\|^2\right)
\\
&+ \frac{5L_{f,1}^2}{n}\sum_{i=1}^n \|v_i^k-\v^k\|^2 + 5 (L_{f,2}r_v+L_{F,1})^2 \|\y^k - y^*(\x^k)\|^2 + 5 L_{f,1}^2 \|\v^k - v^*(\x^k)\|^2,
\end{align*}
which completes the proof.
\end{proof}

Now, combining \eqref{dx}-\eqref{dy}, \eqref{vk+1-vk}, \eqref{xk+1-xk}-\eqref{yk+1-yk}, \eqref{bardy}-\eqref{bardv} with Lemma \ref{tx_bounded} and Remark \ref{tv_bound}, we can establish the boundness of $ \frac{1}{n}\sum_{i=1}^{n}\|t_{v,i}^{k+1}-\t_v^{k+1}\|^2 $, $\frac{1}{n} \sum_{i=1}^{n}\|t_{y,i}^{k+1}-\t_y^{k+1}\|^2 $ and $\frac{1}{n} \sum_{i=1}^{n}\|t_{x,i}^{k+1}-\t_x^{k+1}\|^2 $:
\begin{align}
\frac{1}{n}&\sum_{i=1}^{n}\|t_{y,i}^{k+1}-\t_y^{k+1}\|^2 \leq (\rho+\frac{4C\beta^2}{1-\rho})\frac{1}{n}\sum_{i=1}^{n}\|t_{y,i}^{k}-\t_y^{k}\|^2+\frac{4C\alpha^2}{(1-\rho)n} \sum_{i=1}^{n}\|t_{x,i}^{k}-\t_x^{k}\|^2 \nonumber\\
&+ \frac{C(8+8L_{f,1}^2\beta^2)}{(1-\rho)n}\sum_{i=1}^{n}(\|x_i^k-\x^k\|^2+\|y_i^k-\y^k\|^2) + \frac{4C\alpha^2}{1-\rho}\|\d_x^k\|^2+\frac{8CL_{f,1}^2\beta^2}{1-\rho}\|\y^k-y^*(\x^k)\|^2,\label{ty}\\
\frac{1}{n}&\sum_{i=1}^{n}\|t_{x,i}^{k+1}-\t_x^{k+1}\|^2 \leq (\rho+\frac{4C\alpha^2}{1-\rho})\frac{1}{n}\sum_{i=1}^{n}\|t_{x,i}^{k}-\t_x^{k}\|^2+\frac{4C}{(1-\rho)n} \sum_{i=1}^{n}(\beta^2\|t_{y,i}^{k}-\t_y^{k}\|^2+\eta^2\|t_{v,i}^{k}-\t_v^{k}\|^2) \nonumber\\
&+\frac{C(8+8L_{f,1}^2\beta^2+20L_1^2\eta^2)}{(1-\rho)n}\sum_{i=1}^{n}(\|x_i^k-\x^k\|^2+\|y_i^k-\y^k\|^2) +\frac{C(8+20L_{f,1}^2\eta^2)}{(1-\rho)n}\sum_{i=1}^{n}\|v_i^k-\v^k\|^2\nonumber\\
&+\frac{4C\alpha^2}{1-\rho}\|\d_x^k\|^2+\frac{C(8L_{f,1}^2\beta^2+20L_1^2\eta^2)}{1-\rho}\|\y^k-y^*(\x^k)\|^2+\frac{20CL_{f,1}^2\eta^2}{1-\rho}\|\v^k-v^*(\x^k)\|^2,\label{tx}\\
\frac{1}{n}&\sum_{i=1}^{n}\|t_{v,i}^{k+1}-\t_v^{k+1}\|^2 \leq (\rho+\frac{4C\eta^2}{1-\rho})\frac{1}{n}\sum_{i=1}^{n}\|t_{v,i}^{k}-\t_v^{k}\|^2+\frac{4C}{(1-\rho)n} \sum_{i=1}^{n}(\beta^2\|t_{y,i}^{k}-\t_y^{k}\|^2+\alpha^2\|t_{x,i}^{k}-\t_x^{k}\|^2) \nonumber\\
&+\frac{C(8+8L_{f,1}^2\beta^2+20L_1^2\eta^2)}{(1-\rho)n}\sum_{i=1}^{n}(\|x_i^k-\x^k\|^2+\|y_i^k-\y^k\|^2) +\frac{C(8+20L_{f,1}^2\eta^2)}{(1-\rho)n}\sum_{i=1}^{n}\|v_i^k-\v^k\|^2\nonumber\\
&+\frac{4C\alpha^2}{1-\rho}\|\d_x^k\|^2+\frac{C(8L_{f,1}^2\beta^2+20L_1^2\eta^2)}{1-\rho}\|\y^k-y^*(\x^k)\|^2+\frac{20CL_{f,1}^2\eta^2}{1-\rho}\|\v^k-v^*(\x^k)\|^2,\label{tv}
\end{align}
where $C = \max \{C_1, C_2\}$ and $ L_1 $ is defined in \eqref{def-constants-main-text}.

\subsection{Convergence Rate of Algorithm \ref{alg:slDB}}
Before the final proof, we first establish several useful lemmas.
Recall that $\Phi(x) = F(x,y^*(x))$ denotes the overall objective function. %, we first establish the descent of $\Phi(\x^k)$.
\begin{lemma}\label{dlemma1a}
	The sequence $\{({x}_i^k, {y}_i^k,{v}_i^k)\}$ generated by Algorithm \ref{alg:slDB} satisfies
	\begin{align}
	\label{xf}
	 \Phi(\bar{x}^{k+1})-\Phi(\bar{x}^k)
	\leq & -\frac{\alpha}{2}\|\nabla \Phi (\bar{x}^k)\|^2
	-\frac{1}{2}\left(\frac{1}{\alpha}
	-L_{\Phi}\right)\|\bar{x}^{k+1}-\bar{x}^k\|^2 +\frac{5\alpha L_{f,1}^2}{2}\|\v^{k}-v^* (\x^k)\|^2 \nonumber\\
	&+\frac{5\alpha \left(L_{F,1} + L_{f,2}r_v\right)^2}{2}	\|\y^{k}-y^* (\x^k)\|^2
     +\frac{5\alpha L^2_{f,1}}{2n} \sum_{i=1}^n \|v_i^k-\v^k\|^2 \nonumber\\
	&+\frac{5\alpha\left(L_{F,1}+ r_vL_{f,2}\right)^2}{2n} \sum_{i=1}^n \left(\|x_i^k-\x^k\|^2+ \|y_i^k-\y^k\|^2\right),
	\end{align}
	where $L_{\Phi}$ is defined in \eqref{def-constants-main-text}. 
\end{lemma}
\begin{proof}
	It follows from  \cite[Lemma 2.2]{ghadimi2018approximation} that $\nabla \Phi (x)$ is $L_{\Phi}$-Lipschitz continuous.
On the other hand, from \cite[Lemma 5.7]{beck2017first} we derive
	\begin{align}\label{phi0}
\Phi(\x^{k+1})-\Phi(\x^k)	%= \, &F(\x^{k+1},y^*(\x^{k+1}))-F(\x^{k},y^*(\x^{k})) \nonumber\\
	\leq \, & \big\langle \nabla \Phi(\x^k), \x^{k+1}-\x^k \big\rangle+\frac{L_{\Phi}}{2}\|\x^{k+1}-\x^k\|^2\nonumber\\
	\stackrel{\eqref{xiter}}= \, & -\alpha \big\langle \nabla \Phi (\x^k), \d^k_x \big\rangle +\frac{L_{\Phi}}{2}
	\|\x^{k+1}-\x^k\|^2\nonumber\\
	= \, & -\frac{\alpha}{2}\|\nabla \Phi(\x^k)\|^2
	-\frac{\alpha}{2}
	\|\d^k_x\|^2 +\frac{\alpha}{2}\|\nabla \Phi (\x^k)-\d^k_x\|^2
	+\frac{L_{\Phi}}{2}\|\x^{k+1}-\x^k\|^2 \\
     \stackrel{\eqref{xiter}}=\, & -\frac{\alpha}{2}\|\nabla \Phi(\x^k)\|^2
	-\frac{1}{2\alpha} 	\|\x^{k+1}-\x^k\|^2 +\frac{\alpha}{2}\|\nabla \Phi (\x^k)-\d^k_x\|^2
	+\frac{L_{\Phi}}{2}\|\x^{k+1}-\x^k\|^2. \nonumber
	\end{align}
	By   definition, we have
$\nabla \Phi (\x^k)=\nabla_1 F(\x^k,y^*(\x^k))-\nabla^2_{12}f(\x^k,y^*(\x^k)) v^*(\x^k)$ and
	\begin{align*}
\d^k_x = \frac{1}{n}\sum_{i=1}^n\Big(\nabla_1 F_i(x^k_i, y^k_i) - \nabla^2_{12} f_i(x^k_i, y^k_i) v^k_i\Big).
	\end{align*}
For convenience, we define
\begin{equation*}
\left\{
\begin{array}{l}
\Delta_3 := \left\|\nabla_1 F(\x^k,y^* (\x^k))-\nabla^2_{12}f (\x^k,y^* (\x^k)) v^* (\x^k)-\nabla_1 F(\x^k,\y^k)+\nabla^2_{12}f (\x^k,\y^k) \v^k\right\|, \smallskip \\
\Delta_4 := \left\|\nabla_1 F(\x^k,\y^k)-\nabla^2_{12}f(\x^k,\y^k) \v^k- \frac{1}{n}\sum_{i=1}^n\Big(\nabla_1 F_i(x^k_i, y^k_i) - \nabla^2_{12} f_i(x^k_i, y^k_i) v^k_i\Big)\right\|.
\end{array}
\right.
\end{equation*}
From the triangle inequality, following the similar steps in \eqref{dv1} and \eqref{dv2} respectively, we have
	\begin{align}\label{phi2}
\Delta_3
	\leq \, 
	\left(L_{F,1} + L_{f,2} r_v\right) \|\y^k-y^* (\x^k)\|
	+L_{f,1} \|\v^{k}-v^* (\x^k)\|,
	\end{align}
and 
	\begin{align}\label{phi3}
\Delta_4
	\leq \,  \frac{L_{F,1}+ r_v L_{f,2}}{n}\sum_{i=1}^n \big(\|x_i^k-\x^k\|+ \|y_i^k-\y^k\|\big)
              + \frac{L_{f,1}}{n}\sum_{i=1}^n \|v_i^k-\v^k\|.
	\end{align}
It is easy to observe from the triangle inequality that $\|\nabla \Phi(\x^k)-\d^k_x\| \leq \Delta_3 + \Delta_4$.
Then, by using \eqref{phi2}-\eqref{phi3}, it is easy to derive 
 \begin{align*}
\left\|\nabla \Phi(\x^k)-\d^k_x\right\|^2
     \leq  &\,  5\left(L_{F,1} + L_{f,2} r_v\right)^2 \|\y^k-y^* (\x^k)\|^2 + 5L_{f,1}^2 \|\v^{k}-v^* (\x^k)\|^2 \\
&+\frac{5\left(L_{F,1}+ r_v L_{f,2}\right)^2}{n}\sum_{i=1}^n\left( \|x_i^k-\x^k\|^2+ \|y_i^k-\y^k\|^2\right)
+ \frac{5L_{f,1}^2}{n}\sum_{i=1}^n \|v_i^k-\v^k\|^2,
 \end{align*}
 which, together with \eqref{phi0}, yields the desired result.
\end{proof}

The next two lemmas bound $\Vert\y^{k+1}-y^*(\x^k)\Vert$ and $\Vert \v^{k+1} - v^* (\x^k)\Vert$, respectively.
\begin{lemma}\label{lemmay}
	The sequence $\{(x^k_i,y^k_i,v^k_i)\}$ generated by Algorithm \ref{alg:slDB} satisfies
	\begin{align}\label{y*}
	\Vert\y^{k+1}-y^*(\x^k)\Vert^2&\leq\Big(1-\frac{\beta \sigma}{2}\Big)\|\y^k-y^*(\x^k)\|^2
+\frac{3\beta L_{f,1}^2}{n  \sigma} \sum_{i=1}^n\left(\|\x^k-x^k_i\|^2+\|\y^k-y^k_i\|^2\right).
	\end{align}
\end{lemma}
\begin{proof}
	First, by Cauchy-Schwartz inequality, for any $\xi>0$, we have
	\begin{align}\label{y2p}
\|\y^{k+1} \, - \, & y^*(\x^k)\|^2 =    \|\big[\y^k-\beta \nabla_2f(\x^k,\y^k)-y^*(\x^k)\big] +\beta\big[ \nabla_2f(\x^k,\y^k)- \d_y^k\big] \|^2 \nonumber\\
	\leq \, &(1+\xi) \|\y^k-\beta \nabla_2f(\x^k,\y^k)-y^*(\x^k) \|^2+(1+1/\xi)\beta ^2 \|\nabla_2f(\x^k,\y^k)-\d_y^k \|^2.
	\end{align}
Note that $ \nabla_2f(\x^k,y^*(\x^k))=0$. It follows from the $\sigma$-strong convexity of $f(\bar{x}^k,\cdot)$, $L_{f,1}$-smoothness of $f$ and \cite[Theorem 2.1.12]{nesterov2018lectures}  that
	\begin{align}
	 \langle\y^k-y^*(\x^k),\nabla_2f(\x^k,\y^k) \rangle &=
\langle\y^k-y^*(\x^k),\nabla_2f(\x^k,\y^k)-\nabla_2f(\x^k,y^*(\x^k)) \rangle \nonumber \\
	&\geq \frac{\sigma L_{f,1}}{\sigma+L_{f,1}}\|\y^k-y^*(\x^k)\|^2+\frac{1}{\sigma+L_{f,1}}\|\nabla_2f(\x^k,\y^k)\|^2. \label{jy-03}
	\end{align}
By expanding $\Vert\y^k-\beta \nabla_2f(\x^k,\y^k)-y^*(\x^k)\Vert^2$, plugging in \eqref{jy-03}, and noting that $\beta \leq\frac{2}{\sigma+L_{f,1}}$ in Algorithm \ref{alg:slDB}, we obtain
	\begin{align}\label{y1}
	 \Vert\y^k-\beta \nabla_2f(\x^k,\y^k)-y^*(\x^k)\Vert^2
     \leq  \Big(1- \frac{2\beta \sigma L_{f,1}}{\sigma+L_{f,1}}\Big)\|\y^k-y^*(\x^k)\|^2\leq (1-\beta \sigma)\|\y^k-y^*(\x^k)\|^2.
	\end{align}
It is elementary to show from $f={1\over n}\sum_{i=1}^{n}f_i$, the definition of $\d_y^k$ in \eqref{def:bar9}, the triangle inequality and Assumption \ref{Assump1} (c) that	
\begin{align}\label{y2}
	\Vert\nabla_2f(\x^k,\y^k)-\d_y^k\Vert^2 
	  \leq \frac{L_{f,1}^2}{n}\sum_{i=1}^n\Big(\|\x^k-x^k_i\|^2+\|\y^k-y^k_i\|^2\Big),
\end{align}
	which, together with \eqref{y2p}, \eqref{y1} and the relation $\d^k_{y} = \t^k_{y}$,  yields 
	\begin{align*}
	\Vert \y^{k+1} - \,  y^*(\x^k)\Vert^2  &\leq(1+\xi)(1-\beta\sigma)\|\y^k-y^*(\x^k)\|^2+(1+1/\xi)\frac{\beta^2L_{f,1}^2}{n}\sum_{i=1}^n\big(\|\x^k-x^k_i\|^2+\|\y^k-y^k_i\|^2\big).
	\end{align*}
Finally, we arrive at the desired result \eqref{y*} by taking $\xi =  \beta\sigma/2$ and using $\beta\sigma\leq 1$.
\end{proof}

\begin{lemma}\label{lemmav}
	The sequence  $\{(x^k_i,y^k_i,v^k_i)\}$ generated by Algorithm \ref{alg:slDB} satisfies
\begin{align}\label{v^*}
&\Vert \v^{k+1} - v^* (\x^k)\Vert^2
\leq\left(1- {\eta  \sigma}/{2}\right) \left\|\v^k-v^* (\x^k)\right\|^2+\frac{3\eta \left(L_{F,1}+L_{f,2} r_v\right)^2}{ \sigma}
\left\| \y^k-y^* (\x^k) \right\|^2\nonumber\\
& \ \ +\frac{9\eta\left(L_{F,1}+L_{f,2} r_v\right)^2}{n\sigma} \sum_{i=1}^n \left(\|x_i^k-\x^k\|^2+ \|y_i^k-\y^k\|^2\right)+
\frac{2\eta^2\rho^2}{n}\sum_{i = 1}^{n}\| t_{v,i}^k-\t_v^k\|^2\nonumber\\
&\ \ +\left( {9\eta L_{f,1}^2}/{\sigma} + 2\rho^2\right)\frac{1}{n}\sum_{i=1}^n \|v_i^k-\v^k\|^2.
\end{align}
\end{lemma}
\begin{proof}
For convenience, we define $ \Delta_5 := \v^k+\eta \left[\nabla_2 F(\x^k,\y^{k})- \nabla^2_{22} f(\x^k,\y^{k})\v^k\right]-v^*(\x^k). $ 
Similar to \eqref{y2p}, for any $\delta>0$, we have
	\begin{align}\label{v2p}
\|\v^k+\eta \d_v^k-v^*(\x^k)\|^2 \leq (1+\delta)\left\|\Delta_5\right\|^2 +(1+1/\delta)\eta ^2\Delta_2^2.
	\end{align}
Recall that the term $ \Delta_2 $ is treated in \eqref{dv2}. Next, we treat the term $\|\Delta_5\|^2$ in \eqref{v2p}.
Since $\nabla^2_{22} f(\x^k,y^* (\x^k))v^* (\x^k)=\nabla_2 F(\x^k,y^* (\x^k))$, we have the following reformulation:
$\Delta_5 = \Delta_{5,1} -\eta \big(\Delta_{5,2} + \Delta_{5,3}\big)$, where
\begin{equation*}
\left\{
\begin{array}{l}
\Delta_{5,1} := \left[I - \eta\nabla^2_{22} f(\x^k,\y^k)\right]\Big(\v^k-v^* (\x^k)\Big), \smallskip \\
\Delta_{5,2} := \left[\nabla^2_{22} f(\x^k,\y^k)-\nabla^2_{22} f(\x^k,y^* (\x^k))\right] v^* (\x^k), \smallskip \\
\Delta_{5,3} := \nabla_2 F(\x^k,y^* (\x^k))-\nabla_2 F(\x^k,\y^k).
\end{array}
\right.
\end{equation*}
By Cauchy-Schwartz inequality, for any $\delta_1>0$, we have
	\begin{align*}
	 \|\Delta_5\|^2 	\leq
	\left(1+\delta_1\right) \| \Delta_{5,1} \|^2 + \left(1+1/{\delta_1}\right) \eta^2 \| \Delta_{5,2} + \Delta_{5,3} \|^2 .
	\end{align*}	
Since $\eta\leq \bar{\eta}\leq 1/L_{f,1}$ and $f (x,\cdot)$ is $\sigma$-strongly convex, there holds
	\begin{align*}
	\left\| \Delta_{5,1}\right\|
	\leq   \left\| I-\eta \nabla^2_{22} f (\x^k,\y^k)\right\|_{\textrm{op}}\left\|\v^k-v^* (\x^k)\right\|
    \leq   \left(1-\eta \sigma\right) \left\|\v^k-v^* (\x^k)\right\|.
	\end{align*}
Furthermore, it is apparent from Assumption \ref{Assump1} that
$\left\|\Delta_{5,2}+\Delta_{5,3}\right\| \leq   \left(L_{f,2}r_v + L_{F,1}\right) 	\| \y^k-y^* (\x^k) \|$.
Taking $\delta_1=\eta \sigma$ and noting $\eta\sigma\leq\eta L_{f,1}\leq 1$, we obtain
	\begin{align}\label{vt1}
\left\|\Delta_5\right\|^2
	\leq &
	\left(1+\eta \sigma\right)\left(1-\eta \sigma\right)^2 \left\|\v^k-v^* (\x^k)\right\|^2+\left(1+1/{\eta \sigma}\right) \eta^2\left(L_{f,2}r_v + L_{F,1}\right)^2
	\left\| \y^k-y^* (\x^k) \right\|^2\nonumber\\
	\leq & \left(1-\eta  \sigma\right) \left\|\v^k-v^* (\x^k)\right\|^2+\frac{2\eta \left(L_{f,2}r_v + L_{F,1}\right)^2}{ \sigma}
	\left\| \y^k-y^* (\x^k) \right\|^2.
	\end{align}
We next evaluate $\sum_{i=1}^{n}\Vert v^{k+1}_i - v^* (\x^k)\Vert^2$  to bound $\Vert \v^{k+1} - v^* (\x^k)\Vert^2$.
	By the update of $v^k_i$ in Algorithm \ref{alg:slDB}, we have
	\begin{align*}
	  \sum_{i=1}^{n}\Vert v^{k+1}_i - \, & v^* (\x^k)\Vert^2
    = \sum_{i=1}^{n}\left\| \mathcal{P}_{r_v}\left[\sum\nolimits_{j=1}^{n}w_{ij}\big(v_j^k + \eta t_{v,j}^k\big)\right]- v^* (\x^k)\right\|^2 \\
	\leq\, & \sum_{i=1}^{n}\left\| \sum\nolimits_{j=1}^{n}w_{ij}\big(v_j^k + \eta t_{v,j}^k\big)- v^* (\x^k)\right\|^2\\
	=\, &\sum_{i = 1}^{n}\left\| \sum\nolimits_{j=1}^{n}w_{ij}v_j^k -\v^k+ \eta \Big( \sum\nolimits_{j=1}^{n}w_{ij}t_{v,j}^k-\t_v^k\Big)\right\|^2 + \sum_{i = 1}^{n}\Vert \v^k - v^* (\x^k) + \eta \t_v^k \Vert^2\\
	\leq\, &\sum_{i = 1}^{n}\Vert \v^k - v^* (\x^k) + \eta \t_v^k \Vert^2+2\sum_{i = 1}^{n}\left\| \sum\nolimits_{j=1}^{n}w_{ij}v_j^k -\v^k\right\|^2 + 2\eta ^2\sum_{i = 1}^{n}\left\| \sum\nolimits_{j=1}^{n}w_{ij}t_{v,j}^k-\t_v^k\right\|^2\\
	\leq\, &\sum_{i = 1}^{n}\Vert \v^k - v^* (\x^k) + \eta \t_v^k \Vert^2+2\rho^2\sum_{i=1}^{n}\|v_i^k-\v^k\|^2 + 2\eta ^2\rho^2\sum_{i = 1}^{n}\| t_{v,i}^k-\t_v^k\|^2,
	\end{align*}
	where the second equality holds because	$\sum_{i = 1}^{n}\left[\sum_{j=1}^{n}w_{ij}v_j^k -\v^k+ \eta ( \sum_{j=1}^{n}w_{ij}t_{v,j}^k-\t_v^k)\right]=0$, and the last inequality follows from Lemma \ref{con_sqr}.
Note that $\d^k_{v} = \t^k_{v}$. Then, by combining \eqref{v2p}, \eqref{vt1} and \eqref{dv2}, it can be concluded that
	\begin{align*}
	\Vert \v^{k+1} -\, & v^* (\x^k)\Vert^2 \leq \frac{1}{n}\sum_{i=1}^{n}\Vert v^{k+1}_i - v^* (\x^k)\Vert^2 \\
	&\leq \Vert \v^k - v^* (\x^k) + \eta \t_v^k \Vert^2 + \frac{2\rho^2}{n}\sum_{i=1}^{n}\|v_i^k-\v^k\|^2 + \frac{2\eta^2\rho^2}{n}\sum_{i = 1}^{n}\| t_{v,i}^k-\t_v^k\|^2\\
	&\leq(1+\delta)\left(1-\eta  \sigma\right) \left\|\v^k-v^* (\x^k)\right\|^2+\frac{2(1+\delta)\eta \left(L_{F,1}+L_{f,2} r_v\right)^2}{ \sigma}	\left\| \y^k-y^* (\x^k) \right\|^2\\
	&\ \ + \frac{3(1+1/\delta)\eta^2}{n}\Big[ \left(L_{F,1}+L_{f,2} r_v\right)^2 \sum_{i=1}^n\left( \|x_i^k-\x^k\|^2+ \|y_i^k-\y^k\|^2\right)+ L_{f,1}^2\sum_{i=1}^n \|v_i^k-\v^k\|^2\Big]\\
	& \ \ + \frac{2\rho^2}{n}\sum_{i=1}^{n}\|v_i^k-\v^k\|^2 + \frac{2\eta^2\rho^2}{n}\sum_{i = 1}^{n}\| t_{v,i}^k-\t_v^k\|^2.
	\end{align*}
Finally, the   desired result \eqref{v^*} follows by setting $\delta= \eta\sigma/2$ and using $\eta\sigma\leq 1$.
\end{proof}

\begin{lemma}\label{*k+1-*k}
Let $L_v$ be defined in \eqref{def-constants-main-text}. There hold
\begin{align}\label{y*k+1-y*k}
\|y^*(\x^{k+1})-y^*(\x^{k})\| \leq \frac{L_{f,1}}{\sigma} \|\x^{k+1}-\x^{k}\|
\text{~~and~~}
\|v^*(\x^{k+1})-v^*(\x^{k})\| \leq \frac{L_v}{\sigma} \|\x^{k+1}-\x^{k}\|.
\end{align}
\end{lemma}
\begin{proof}
Due to the optimality of $y^*(x)$, we have $\nabla_2 f(x,y^*(x))=0$ for any $x$.
Now, let $x$ and $x'$ be arbitrarily fixed.
Then, it follows from the $\sigma$-strong convexity of $f(x,\cdot)$ and $L_{f,1}$-Lipschitz continuity of $\nabla f$ that
\begin{align*}
\sigma\|y^*(x)-y^*(x')\|&\leq \|\nabla_2 f(x,y^*(x))-\nabla_2 f(x,y^*(x'))\|
\\&=  \|\nabla_2 f(x,y^*(x'))-\nabla_2 f(x',y^*(x'))\|\leq L_{f,1}\|x-x'\|
\end{align*}
Hence, we have
\begin{align}\label{y*(x)-y*(x')}
\|y^*(x)-y^*(x')\| \leq (L_{f,1}/\sigma)\|x-x'\|.
\end{align}
Then we can get first inequality immediately by taking $ x=\x^{k+1} $ and $ x'=\x^k $.
Next, to obtain the second inequality, we define
\begin{equation*}
\left\{
\begin{array}{l}
\Delta_6 := \nabla_2 F(\x^k,y^* (\x^k)) - \nabla_2 F(\x^{k+1},y^* (\x^{k+1})), \smallskip \\
\Delta_7 := \big[\nabla^2_{22} f(\x^{k+1},y^* (\x^{k+1}))-\nabla^2_{22} f(\x^k,y^* (\x^k))\big]v^* (\x^{k+1}).
\end{array}
\right.
\end{equation*}
By using Assumption \ref{Assump1}, we have
\begin{equation}\label{jy-05}
\left\{
\begin{array}{l}
\|\Delta_6\| \leq L_{F,1}\big(\|\x^{k+1}-\x^k\|+\|y^* (\x^{k+1})-y^* (\x^k)\|\big), \smallskip \\
\|\Delta_7\| \leq L_{f,2}r_v\big(\|\x^{k+1}-\x^k\|+\|y^* (\x^{k+1})-y^* (\x^k)\|\big).
\end{array}
\right.
\end{equation}
It follows from $\nabla^2_{22} f(\x^k,y^* (\x^k))v^* (\x^k)=\nabla_2 F(\x^k,y^* (\x^k))$ that
\begin{align}\label{jy-04}
\nabla^2_{22} f(\x^k,y^* (\x^k))(v^* (\x^k)-v^* (\x^{k+1})) = \Delta_6 + \Delta_7.
\end{align}
The $\sigma$-strong convexity of $f(\x^k,\cdot)$ implies that
\begin{align*}
\sigma\|v^* (\x^k)& - v^* (\x^{k+1})\|
 \leq \|\nabla^2_{22} f(\x^k,y^* (\x^k))(v^* (\x^k)-v^* (\x^{k+1}))\|  \\
& \stackrel{\eqref{jy-04}}= \|\Delta_6 + \Delta_7\|  \leq \|\Delta_6\| + \|\Delta_7\| 
  \stackrel{(\ref{jy-05}, \ref{y*(x)-y*(x')})}\leq  \left(L_{F,1}+L_{f,2}r_v\right)\left(1+ {L_{f,1}}/{\sigma} \right)\|\x^{k+1}-\x^k\|,
\end{align*}
which implies the second inequality in \eqref{y*k+1-y*k} by noting the definition of $L_v$ in \eqref{def-constants-main-text}.
\end{proof}

Now, we are ready to prove Theorem \ref{cr}.
\begin{proof}
By using Cauchy-Schwartz inequality again, we derive
\begin{align*}
\|\y^{k+1}-y^*(\x^{k+1})\|^2\leq \left(1+\frac{\beta\sigma}{4}\right)\|\y^{k+1}-y^*(\x^{k})\|^2+ \left(1+\frac{4}{\beta\sigma}\right)\|y^*(\x^{k+1})-y^*(\x^{k})\|^2.
\end{align*}
Taking into account \eqref{y*}, the first inequality in \eqref{y*k+1-y*k}, and $\beta\sigma\leq 1$, we obtain
\begin{align}\label{yf}
&\|\y^{k+1}-y^*(\x^{k+1})\|^2 
\leq \, \Big(1-\frac{\beta \sigma}{4}\Big)\|\y^k-y^*(\x^k)\|^2 \nonumber\\
&\quad +\frac{15\beta L_{f,1}^2}{4 n\sigma} \sum_{i=1}^n\left(\|\x^k-x^k_i\|^2+\|\y^k-y^k_i\|^2\right) +\frac{5L_{f,1}^2}{\beta\sigma^3}\|\x^{k+1}-\x^{k}\|^2.
\end{align}
Using Cauchy-Schwartz inequality again, we derive
\begin{align*}
\|\v^{k+1}-v^*(\x^{k+1})\|^2\leq \left(1+\frac{\eta\sigma}{4}\right)\|\v^{k+1}-v^*(\x^{k})\|^2+ \left(1+\frac{4}{\eta\sigma}\right)\|v^*(\x^{k+1})-v^*(\x^{k})\|^2.
\end{align*}
Similarly, taking into account \eqref{v^*}, the second inequality in \eqref{y*k+1-y*k}, and $\eta\sigma\leq 1$, we obtain
\begin{align}\label{vf}
\|\v^{k+1}- \, & v^*(\x^{k+1})\|^2 
\leq  \left(1-\frac{\eta  \sigma}{4}\right) \left\|\v^k-v^* (\x^k)\right\|^2+\frac{15\eta \left(L_{F,1}+L_{f,2} r_v\right)^2}{4\sigma}
\left\| \y^k-y^* (\x^k) \right\|^2\nonumber\\
& +\frac{5L_v^2}{\eta\sigma^3}\|\x^{k+1}-\x^{k}\|^2 +\frac{45\eta\left(L_{F,1}+L_{f,2} r_v\right)^2}{4n\sigma}\sum_{i=1}^n \left(\|x_i^k-\x^k\|^2+ \|y_i^k-\y^k\|^2\right)\nonumber\\
&  + \left(\frac{45\eta L_{f,1}^2}{4\sigma}+\frac{5\rho^2}{2}\right)\frac{1}{n}\sum_{i=1}^{n}\|v_i^k-\v^k\|^2 + \frac{5\rho^2\eta ^2}{2n} \sum_{i = 1}^{n}\| t_{v,i}^k-\t_v^k\|^2.
\end{align}
Define the Lyapunov function as in \eqref{Lyapunov}. By combining
 \eqref{tx}-\eqref{tv}, \eqref{x-xbar}-\eqref{v-vbar}, \eqref{xf}, \eqref{yf}, \eqref{vf}, it is straightforward to derive
\begin{align}\label{V}
V_{k+1} &- V_k \leq -\frac{\alpha}{2}\|\nabla \Phi (\bar{x}^k)\|^2
-A_1\|\d_x^k\|^2 - A_2 \|\y^k - y^*(\x^k)\|^2 -A_3  \|\v^k - v^*(\x^k)\|^2\\
&-  \frac{A_4}{n} \sum_{i=1}^n \|x_i^k - \x^k\|^2 - \frac{A_5 }{n} \sum_{i=1}^n \|y_i^k - \y^k\|^2 - \frac{A_6}{n} \sum_{i=1}^n \|v_i^k - \v^k\|^2\\
& -  \frac{A_7}{n} \sum_{i=1}^n \|t_{x,i}^k - \t_x^k\|^2 -  \frac{A_8}{n} \sum_{i=1}^n \|t_{y,i}^k - \t_y^k\|^2 -  \frac{A_9}{n} \sum_{i=1}^n \|t_{v,i}^k - \t_v^k\|^2,
\end{align}
where the coefficients are given by
\small
\begin{align*} 
A_1 &= \frac{\alpha}{2}-\frac{L_{\Phi}\alpha^2}{2} -\frac{5L_{f,1}^2 a_1 \alpha^2}{\sigma^3 \beta} -\frac{5L_v^2 a_2 \alpha^2}{\sigma^3 \eta} - \frac{4C ( a_6 \alpha^2 +  a_7\beta^2 + a_8\eta^2) \alpha^2}{1 - \rho},\\
A_2 &= \frac{\sigma a_1\beta}{4}-\frac{5 L_1^2 \alpha}{2}-\frac{15 L_1^2 a_2 \eta}{4\sigma} - \frac{C( a_6 \alpha^2 +  a_8\eta^2)(8 L_{f,1}^2 \beta^2 + 20L_1^2\eta^2)}{1 - \rho} - \frac{8 C L_{f,1}^2 a_7 \beta^4}{1-\rho},\\
A_3 &= \frac{\sigma a_2 \eta}{4}-\frac{5 L_{f,1}^2 \alpha}{2} - \frac{20C L_{f,1}^2 \eta^2( a_6\alpha^2 +  a_8\eta^2) }{1-\rho}, \\
A_4 &= a_3 (1-\rho) - \frac{5 L_1^2 \alpha}{2}-\frac{15 L_{f,1}^2 a_1 \beta }{4\sigma}-\frac{45 L_1^2 a_2 \eta}{4\sigma} - \frac{C( a_6\alpha^2 +  a_8\eta^2)(8 + 8 L_{f,1}^2 \beta^2 + 20 L_1^2 \eta^2)}{1 - \rho} - \frac{C a_7\beta^2 (8+ 8 L_{f,1}^2\beta^2)}{1-\rho},\\
A_5 &=a_4 (1-\rho) - \frac{5 L_1^2 \alpha}{2} - \frac{15 L_{f,1}^2 a_1 \beta }{4\sigma} - \frac{45 L_1^2 a_2 \eta}{4\sigma}  - \frac{C( a_6\alpha^2 +  a_8\eta^2)(8 + 8 L_{f,1}^2 \beta^2 + 20 L_1^2 \eta^2)}{1 - \rho} - \frac{C a_7\beta^2 (8+ 8 L_{f,1}^2\beta^2)}{1-\rho},\\
A_6 &= a_5 (1-\rho)- \frac{5 L_{f,1}^2 \alpha }{2}-\frac{45 L_1^2 a_2 \eta }{4\sigma} -\frac{5\rho^2 a_2}{2} - \frac{C( a_6\alpha^2 +  a_8\eta^2)(8 + 20 L_{f,1}^2 \eta^2)}{1 - \rho},\\
A_7 &=a_6\alpha^2(1-\rho) -\frac{\rho^2 a_3 \alpha^2}{1-\rho} - \frac{4C( a_6\alpha^2 +  a_8\eta^2+  a_7\beta^2)\alpha^2}{1-\rho},\\
A_8 &= a_7\beta^2(1-\rho) -\frac{\rho^2 a_4 \beta^2}{1-\rho} - \frac{4C( a_6\alpha^2 +  a_8\eta^2+  a_7\beta^2)\beta^2}{1-\rho},\\
A_9 &= a_8\eta^2(1-\rho) -\frac{\rho^2 a_5 \eta^2}{1-\rho}-\frac{5\rho^2a_2\eta^2}{2}- \frac{4C( a_6\alpha^2 +  a_8\eta^2)\eta^2}{1-\rho}.
\end{align*}
\normalsize
Here, constants such as $L_{\Phi}$, $L_v$ and $L_1$ are defined in \eqref{def-constants-main-text}, take
\begin{align*}
a_1=a_2=a_3=a_4=1,\, a_5= \frac{10\rho^2}{1-\rho},\,
a_6 =a_7 = \frac{2\rho^2}{(1-\rho)^2},\, a_8 = \max \left\{\frac{30\rho^4}{(1-\rho)^3},\frac{15\rho^2}{2(1-\rho)}\right\},
\end{align*}
and choose stepsizes satisfying
\small
\begin{align}\label{eq:stepsize}
&\beta < \min \left \{\frac{4 \sigma (1-\rho)}{75 L_{f,1}^2 }, \frac{1}{L_{f,1}},\frac{1-\rho}{\sqrt{96 Ca_7}}, \frac{\sigma(1-\rho)}{128C a_7},\frac{1-\rho}{\sqrt{24C}}\right \},\nonumber\\
&\eta < \min \left \{\frac{4 \sigma (1-\rho)}{225 L_1^2 }, \frac{2 \sigma \rho^2}{9 L_1^2 }, \frac{\sigma \beta}{60 L_1^2}, \frac{1}{L_{f,1}},\frac{1}{L_1}, \frac{\rho(1-\rho)}{\sqrt{360Ca_8}},\sqrt{\frac{\sigma(1-\rho)\beta}{896Ca_8}},\frac{\sigma(1-\rho)}{320Ca_8},\frac{1-\rho}{\sqrt{24Ca_8/a_6}},\frac{1-\rho}{\sqrt{24C}}\right\},\\
&\alpha < \min \left \{ \frac{1}{4L_{\Phi}}, \frac{\sigma^3 \beta}{40 L_{f,1}^2}, \frac{\sigma^3 \eta}{40 L_v^2}, \frac{\sigma \beta}{40 L_1^2}, \frac{\sigma \eta}{20 L_{f,1}^2}, \frac{2(1-\rho)}{25 L_1^2}, \frac{\rho^2}{L_{f,1}^2},1, \frac{\rho(1-\rho)}{\sqrt{360Ca_6}},\sqrt{\frac{\sigma(1-\rho)\beta}{896Ca_6}}, \sqrt{\frac{\sigma(1-\rho)\eta}{320Ca_6}},\frac{(1-\rho)^{3/2}}{\sqrt{24C}}\right\},\nonumber
\end{align}
\normalsize 
 it is elementary to show that $ A_1,A_2,\ldots,A_9 $ are all nonnegative. Then it follows that
\begin{align*}
V_{k+1} - V_k \leq -\frac{\alpha}{2}\|\nabla \Phi (\bar{x}^k)\|^2.
\end{align*}
By telescoping,
\begin{align*}
\min_{0 \le k \le K-1}  \|\nabla \Phi(\x^k)\|^2  \leq \frac{1}{K}\sum_{k=0}^{K-1}\|\nabla \Phi(\x^k)\|^2
\leq \frac{2V_0}{\alpha K} = O\left(\frac{1}{K}\right).
\end{align*}
The consensus error can also be established by
\begin{align*}
V_{k+1} - V_k \leq -  \frac{A_4}{n} \sum_{i=1}^n \|x_i^k - \x^k\|^2,
\end{align*}
which yields
\begin{align*}
\min_{0 \le k \le {K-1}}\frac{1}{n}\sum_{i=1}^{n} \|x_i^k - \x^k\|^2 \leq\frac{1}{nK}\sum_{k=0}^{K-1}\sum_{i=1}^{n} \|x_i^k - \x^k\|^2 \leq \frac{V_0}{A_4 K} = O\left(\frac{1}{K}\right).
\end{align*}
Similarly, we can derive
\begin{align*}
\min_{0 \le k \le {K-1}}\frac{1}{n} \sum_{i=1}^n \|y_i^k - \y^k\|^2= O\left(\frac{1}{K}\right) \text{~~~and~~~}\min_{0 \le k \le {K-1}}\frac{1}{n}\sum_{i=1}^n \|v_i^k - \v^k\|^2= O\left(\frac{1}{K}\right),
\end{align*}
which completes the proof of Theorem \ref{cr}.
\end{proof}

\end{document}